\documentclass[11pt]{amsart}

\usepackage[latin1]{inputenc}
\usepackage{amssymb,amsmath,amsfonts}
\usepackage{enumerate}
\usepackage[noadjust]{cite}

\usepackage[a4paper,twoside,left=3.5cm,right=3.2cm,top=3.3cm,bottom=2.6cm]{geometry}
\usepackage{times}
\usepackage{hyperref}

\newtheorem{proposition}{Proposition}[section]
\newtheorem{lemma}[proposition]{Lemma}
\newtheorem{theorem}[proposition]{Theorem}
\newtheorem{corollary}[proposition]{Corollary}

\newcounter{theor}

\newtheorem{teor}[theor]{Theorem}

\theoremstyle{definition}
\newtheorem{definition}[proposition]{Definition}
\newtheorem{remark}[proposition]{Remark}
\newtheorem{example}[proposition]{Example}

\DeclareMathOperator{\Val}{{\bf Val}}
\DeclareMathOperator{\spa}{span}
\DeclareMathOperator{\SL}{SL}
\DeclareMathOperator{\GL}{GL}
\DeclareMathOperator{\SO}{SO}
\renewcommand{\O}{\mathrm{O}}

\newcommand{\R}{\mathbb{R}}
\newcommand{\N}{\mathbb{N}}
\newcommand{\sign}{\mathrm{sign}}
\newcommand{\K}{\mathcal{K}}

\newcommand{\sfe}{\mathbb{S}^{n-1}}
\newcommand{\vol}{V_n}

\DeclareMathOperator{\MVal}{\mathrm{{\bf MVal}}}
\newcommand{\st}{s}

\newcommand{\A}{\mathcal{A}}

\newcommand{\di}{\Phi} 
\newcommand{\dib}{\Psi} 
\newcommand{\rv}{\mu} 
\newcommand{\vv}{\phi} 
\newcommand{\gv}{\varphi} 

\title{Minkowski valuations under volume constraints}

\author{Judit Abardia-Ev\'equoz} 
\address{Institut f\"ur Mathematik, Goethe-Universit\"at Frankfurt, 
Robert-Mayer-Str.~10, 60629 Frankfurt am Main, Germany}
\email{abardia@math.uni-frankfurt.de}

\author{Andrea Colesanti} 
\address{Dipartimento di Matematica ``U. Dini'', Viale Morgagni 67/1, 50134 Firenze, Italy}
\email{colesant@math.unifi.it}

\author{Eugenia Saor\'in G\'omez} 
\address{Fakult\"at f\"ur Mathematik, Otto-von-Guericke Universit\"at Magdeburg, 
Universit\"atsplatz~2, 39106 Magdeburg, Germany}
\email{eugenia.saorin@ovgu.de}

\begin{document}

\thanks{The first author is supported by the DFG grants AB 584/1-1 and AB 584/1-2. 
The second author is supported by the FIR project 2013 ``Geometrical and Qualitative aspects of PDEs'', and by the GNAMPA.
The third author is supported by 19901/GERM/15, Fundaci\'on S\'eneca, CARM, Programa de Ayudas a Grupos de Excelencia de la Regi\'on de Murcia.}

\date{\today}



\begin{abstract}
We provide a description of the space of continuous and translation invariant Minkowski valuations $\di:\K^n\to\K^n$ for which there is an upper and a lower bound for the volume of $\di(K)$ in terms of the volume of the convex body $K$ itself. Although no invariance with respect to a group acting on the space of convex bodies is imposed, we prove that only two types of operators appear: a family of operators having only cylinders over $(n-1)$-dimensional convex bodies as images, and a second family consisting essentially of 1-homogeneous operators. Using this description, we give improvements of some known characterization results for the difference body.
\end{abstract}

\maketitle

\section{Introduction}
An inequality between two geometric quantities associated to a convex body is called \emph{affine isoperimetric inequality} if the ratio of these two quantities is invariant under the action of all affine transformations of the convex body.
Affine isoperimetric inequalities have always constituted an important part of convex geometry and have found numerous applications to different areas, such as functional analysis, partial differential equations, or geometry of numbers (see \cite{lutwak.handbook}). Moreover, affine isoperimetric inequalities are usually stronger than their Euclidean counterparts.

Three of the best known affine isoperimetric inequalities associated to operators between convex bodies are: the Rogers-Shephard inequality, associated to the difference body; the Busemann-Petty centroid inequality, associated to the centroid body; and the Petty projection and Zhang inequalities, associated to the projection body. 
One of the first and most relevant applications of these inequalities was given by Zhang \cite{zhang}, who obtained an affine version of the Sobolev inequality from (an extension of) the Petty projection inequality. Ten years later, Haberl and Schuster \cite{haberl.schuster1,haberl.schuster2} generalized it to an asymmetric affine $L_p$-Sobolev inequality by using the characterization of the $L_p$-projection bodies previously obtained by Ludwig \cite{ludwig05} in the context of the so-called $L_p$-Minkowski valuations. For further results in this direction we refer to \cite[Section 10.15]{schneider.book14}, \cite{haberl.schuster.xiao,ludwig.xiao.zhang,lutwak86,cianchi_lyz_2009,lyz_2010,lyz_2000,lyz_2002,haddad.jimenez.montenegro,wang}, and references therein.

In the present paper, we initiate a study aiming at a deeper understanding of the relationship between affine isoperimetric inequalities and characterization results for Minkowski valuations, by taking the converse direction of Haberl and Schuster~\cite{haberl.schuster1} and classifying, given an affine isoperimetric inequality, all continuous (and translation invariant) Minkowski valuations by which it is satisfied. In this paper, we focus on the affine isoperimetric inequality associated to the difference body operator. 

We denote by $\K^n$ the space of convex and compact sets (convex bodies) in $\R^n$.
The \emph{difference body operator} $D:\K^n\longrightarrow\K^n$ is defined by 
\begin{equation}\label{diff_body}
DK := K + (-K),
\end{equation}
where $-K:=\{x\in\R^n\,:\,-x\in K\}$ and $+$ denotes the Minkowski or vectorial sum.  Notice that the ratio
$$
\frac{V_n(DK)}{V_n(K)}
$$
is invariant under affine transformations of $\R^n$ (here $V_n$ denotes the $n$-dimensional volume).  
The affine isoperimetric inequalities associated to the difference body read as follows
\begin{equation}\tag{RS}
2^n\vol(K)\leq\vol(DK)\leq\binom{2n}{n}\vol(K),\quad\forall K\in\K^n. 
\end{equation}
For convex bodies with non-empty interior, equality holds in the upper inequality exactly if $K$ is a simplex, and convex bodies symmetric with respect to the origin (i.e., $K=-K$) are the only optimizers of the lower inequality. 
The lower bound follows from a direct application of the Brunn-Minkowski inequality (see \cite{schneider.book14}) and the upper bound was proved by Rogers and Shephard in \cite{rogers.shephard} (see also \cite{chakerian,rogers.shephard58} for other proofs and related inequalities). 

We study in this paper the operators $\Phi:\K^n\longrightarrow\K^n$ satisfying an (RS) type inequality, that is, operators such that the volume of the image of a convex body $K$ is bounded uniformly, from above and from below, by a multiple of the volume of $K$ (see Definition~\ref{def_VC_intro}). 
We will always assume these operators to be continuous and Minkowski valuations. 

\smallskip
In the framework of convex geometry, an operator $Z:\K^n\longrightarrow(\A,+)$ is a \emph{valuation} if 
\begin{equation}\label{def_val}
Z(K)+Z(L)=Z(K\cup L)+Z(K\cap L)
\end{equation}
for every $K,L\in\K^n$ such that $K\cup L\in\K^n$. Here $(\A,+)$ denotes an Abelian semigroup.

Valuations have developed as particularly important objects in convex geometry since Dehn's solution to the Third Hilbert problem. Probably the best known result in the theory of valuations is the characterization by Hadwiger~\cite{hadwiger} of the intrinsic volumes as a basis of the space of continuous and motion invariant valuations taking values in $\R$. 
The interested reader is referred to \cite{alesker.survey,bernig.survey,fu.survey,mcmullen93,mcmullen.schneider} 
and \cite[Chapter 6]{schneider.book14} for valuable and detailed surveys about the state of the art of the theory of real-valued valuations. We refer also to \cite{alesker.faifman,bernig.fu.solanes,bernig.fu,ludwig.reitzner} for further recent results in this area.
Apart from the real-valued case, there has been an increasing interest in valuations having other spaces as codomain. Examples of these are, among others, the space of matrices, tensors, area and curvature measures, and various function spaces (see e.g.\! \cite{bernig.hug,haberl.parapatits,ludwig_matrix,ludwig_sob,wannerer1,wannerer2}). 

\emph{Minkowski valuations} are those taking values in $\K^n$ endowed with the Minkowski addition. In other words, $\di:\K^n\longrightarrow\K^n$ is a Minkowski valuation if 
\eqref{def_val} holds for every pair of convex bodies $K$ and $L$ such that $K\cup L\in\K^n$, and  the Minkowski addition is taken on both sides of the equality. As stated before, they will be the main object of study of this work.  

One of the most pursued scopes in the theory of valuations amounts to characterize classical and new objects from the realm of convex geometry, as valuations with specific additional properties. These additional properties are usually of two types:
\begin{itemize}
\item[(i)] topological: continuity or semi-continuity with respect to the standard topology on $\K^n$;
\item[(ii)] algebraic-geometrical: covariance or contravariance with respect to the action of some group of transformations of $\K^n$, such as the group of 
translations, $\GL(n)$ or $\SO(n)$.
\end{itemize}

In this paper we aim to use a property of a different nature: a {\it metric-geometrical} property, namely the fulfillment of the following volume constraint, recently introduced in \cite{abardia.colesanti.saorin1} (see also \cite{abardia.saorin2,colesanti.hug.saorin}).

\begin{definition}\label{def_VC_intro}
Let $\di:\K^n\longrightarrow\K^n$ be an operator. We say that $\di$ \emph{satisfies a  volume constraint (VC)} if there are constants $c_\di,C_\di>0$ such that
\begin{equation}\tag{VC}
c_\di\vol(K)\leq \vol(\di(K))\leq C_\di\vol(K),\quad\forall K\in\K^n.
\end{equation}
\end{definition}

In \cite{abardia.saorin2}, the following characterization result for the difference body operator was obtained, based on (VC).

\begin{teor}[\cite{abardia.saorin2}]\label{thm_as} Let $n\geq 2$. An operator $\di: \K^n\longrightarrow\K^n$ is continuous, $\GL(n)$-covariant, and satisfies the upper bound in the (VC) condition if and only if there are $a,b\geq 0$ such that $\di(K)=aK+b(-K)$ for every $K\in\K^n$.
\end{teor}

This theorem belongs to a very recent and rapidly developing theory of classification results in convex geometry, without the notion of Minkowski valuation but under other natural, and very general, properties such as symmetrization. Some of these general results yield as a corollary a characterization of the difference body operator. We highlight in the following the first of them and refer the reader to \cite{bianchi.gardner.gronchi,gardner.hug.weil1,gardner.hug.weil2,milman.rotem} for more details. 
\begin{teor}[\cite{gardner.hug.weil1}]\label{thm_ghw} Let $n\geq 2$. An operator $\di: \K^n\longrightarrow\K^n$ is continuous, translation invariant, and $\GL(n)$-covariant if and only if there is a $\lambda\geq 0$ such that $\di(K)=\lambda DK$ for every $K\in\K^n$.
\end{teor}

We would like to stress that Theorem~\ref{thm_ghw} does not require the property of being a Minkowski valuation, but requires $\GL(n)$-covariance. 
We recall that $\di$ is said to be {\em covariant} with respect to a group $G$ of transformations of 
$\R^n$ if
$$\di(g(K))=g(\di(K)),\quad\forall\, K\in\K^n,\;\forall\, g\in G.$$

The first works about characterization of Minkowski valuations were obtained by Schneider in \cite{schneider74} and \cite{schneider74.dim2}. He obtained significant classification results for a special type of Minkowski valuations, called {\em Minkowski endomorphism}, which are defined as the  continuous Minkowski valuations that are homogeneous of degree 1, commute with rotations, and are translation invariant. The difference body operator constitutes the fundamental example of a Minkowski endomorphism. 

Ludwig's works \cite{ludwig02, ludwig05} represent the starting point for a systematic study of characterization results in the theory of Minkowski valuations. 
Concerning the difference body operator, she obtained the following fundamental characterization result.
\begin{teor}[\cite{ludwig05}]\label{t: ludwig class DK}Let $n\geq 2$. An operator $\di: \K^n\longrightarrow\K^n$ is a continuous, translation invariant, and 
$\SL(n)$-covariant Minkowski valuation if and only if there is a $\lambda\geq 0$ such that $\di(K)=\lambda DK$ for every $K\in\K^n$.
\end{teor}

After the seminal results of Ludwig, an intensive investigation of Minkowski valuations has been launched, which has led to characterization results, for other groups of transformations or for certain subfamilies of $\K^n$. The corresponding results can be found in  
\cite{abardia.bernig,abardia,boroczky.ludwig,kiderlen,schuster.wannerer.smooth,schuster.wannerer16,haberl,schuster.10,schuster.wannerer12,wannerer.equiv} and references therein.

\subsection{Results of the present paper}
We denote by $\MVal$ the space of continuous and translation invariant Minkowski valuations and by $\MVal^s$ 
the subspace of $\MVal$ consisting of Minkowski valuations with symmetric image. The Steiner point of $K$ is denoted by $\st(K)$. We refer the reader to Section~\ref{preliminaries} for further notation and definitions.

As described above, in the present paper, we consider the general question of describing the operators in $\MVal$ satisfying the (VC) condition {\it without any further hypothesis}. Our main result can be stated as follows:

\begin{theorem}\label{teo}
Let $n\geq 2$ and consider $\di\in\MVal$ satisfying (VC). 
Then exactly one of the following possibilities occurs:
\begin{enumerate}\itemsep10pt
\item[\emph{(i)}] there exist $\di_1\in\MVal$ homogeneous of degree 1 and $p:\K^n\longrightarrow\R^n$ continuous and translation invariant valuation 
such that
$$
\di(K)=\di_1(K)+p(K),\quad\forall\,K\in\K^n;
$$
\item[\emph{(ii)}] there exist a segment $S$, an $(n-1)$-dimensional convex body $L$ with $\dim (L+S)=n$, and $p:\K^n\longrightarrow\R^n$ continuous and 
translation invariant valuation such that 
$$
\di(K)= L+\vol(K)S+p(K),\quad\forall\,K\in\K^n.
$$
\end{enumerate}
\end{theorem}
In Sections~\ref{n-1} and~\ref{1} the operator $p:\K^n\longrightarrow\R^n$ is described more explicitly, as a sum of valuations with fixed degree. 

If we additionally assume that the operator $\di$ has symmetric images, then $p(K)$ is the origin, for every $K\in\K^n$, and we obtain the following.
\begin{theorem}\label{cor}
Let $n\geq 2$ and consider $\di\in\MVal^s$ satisfying (VC). 
Then exactly one of the following possibilities occurs:
\begin{enumerate}\itemsep10pt
\item[\emph{(i)}] $\di$ is homogeneous of degree one;
\item[\emph{(ii)}] there exist a centered segment $S$ and an $o$-symmetric $(n-1)$-dimensional convex body $L$ with $\dim (L+S)=n $ such that  
$$\di(K)= L+\vol(K)S,\quad\forall K\in\K^n.$$
\end{enumerate}
\end{theorem}
We would like to remark that Theorem~\ref{teo} constitutes, to the best of our knowledge, the first characterization result in the theory of Minkowski valuations which does not assume the operator to be invariant, covariant or contravariant with respect to some subgroup of $\GL(n)$. 

A description of the 1-homogeneous Minkowski valuations appearing in Theorem~\ref{cor}(i) was given in~\cite{abardia.colesanti.saorin1} in the context of Minkowski additive operators (i.e.,  continuous, 1-homoge\-ne\-ous, and translation invariant Minkowski valuations). There, the Minkowski endomorphisms satisfying the (VC) condition, and the Minkowski additive operators that satisfy (VC) and are monotonic were classified. 

Theorem~\ref{teo} allows us to improve these results by removing the homogeneity hypothesis and obtain the following.
\begin{theorem}\label{+mon intro2}
Let $n\geq 2$. An operator $\di\in\MVal$ satisfies (VC) and is monotonic if and only if exactly one of the following possibilities occurs:
\begin{enumerate}
\item[\emph{(i)}] there are $g\in\GL(n)$ and $p\in\R^n$ such that $\di(K)=gDK+p$ for every $K\in\K^n$;
\item[\emph{(ii)}] there are $L,S\in\K^n$ with $0\in S$, $\dim S=1$, $\dim L=n-1$, and $\dim(L+S)=n$ such that $\di(K)=L+\vol(K)S$ for every $K\in\K^n$.
\end{enumerate}
\end{theorem}

\begin{theorem}\label{+On_dim_geq_32}
Let $n\geq 3$. 
\begin{enumerate}
\item[\emph{(i)}] An operator $\di\in\MVal$ satisfies (VC) and is $\SO(n)$-covariant if and only if there are $a,b\geq 0$ with $a+b>0$ such that 
$$
\di(K)=a(K-\st(K))+b(-K+\st(K)),\quad\forall K\in\K^n.
$$ 
\item[\emph{(ii)}] An operator $\di\in\MVal^{s}$ satisfies (VC) and is $\SO(n)$-covariant if and only if there is a $\lambda>0$ such that $\di(K)=\lambda DK$ for every $K\in\K^n$.
\end{enumerate}
\end{theorem}

\smallskip
To prove our results we rely upon recent developments from the theory of real-valued valuations, which will be recalled in Section~\ref{preliminaries} for the reader's convenience. 
In addition, we need to develop new techniques, since our assumption of satisfying (VC) is of a different nature than typical covariance or contravariance with respect to some subgroup of $\GL(n)$. 
For Theorem~\ref{teo} we perform a careful study of the image of zonotopes under $\di:\K^n\longrightarrow\K^n$, since the lack of covariance does not allow us to use the standard technique of exploiting the image of few simplices. 
For the proof of Theorems~\ref{+mon intro2} and \ref{+On_dim_geq_32}, we use Theorem~\ref{teo} and classical results in the theory of real-valued valuations.

\smallskip
The paper is organized as follows. In Section~\ref{preliminaries} we collect the known results, especially about valuations, that will be used along the paper, and we introduce the notation used throughout. In Sections~\ref{sec: dependence a point} to \ref{1} we prove Theorem~\ref{teo}. More precisely, Section~\ref{sec: dependence a point} is devoted to show that we have actually a dichotomy under the hypotheses of Theorem~\ref{teo}. This leads to either the 1-homogeneous case, or to case (ii)
of Theorem \ref{teo}. In the next two sections we study each case, giving the proof of Theorems~\ref{teo} and~\ref{cor} in Section~\ref{1}. Finally, in Section~\ref{on_mon}, we prove Theorem~\ref{+mon intro2} and Theorem~\ref{+On_dim_geq_32} together with its analogue for $n=2$.
We end the paper with some examples to illustrate the necessity of our assumptions in Theorem~\ref{teo}.

\section{Preliminaries}\label{preliminaries}
\subsection{Notation} 
As usual, we denote by $\R^n$ the $n$-dimensional Euclidean space, equipped with the standard scalar product $\langle\cdot,\cdot\rangle$.

If $A\subset\R^n$ is a measurable set, $V_n(A)$ denotes its volume, that is, its $n$-dimensional Lebesgue measure. If $A\subset\R^n$, the \emph{span} of $A$, $\spa A$, is the vector subspace of $\R^n$ parallel to the minimal affine subspace in $\R^n$ containing $A$. The \emph{dimension of} $A$ is defined as $\dim A:=\dim (\spa A)$. 

The unit sphere of $\R^n$ is denoted by $\sfe$ and we denote by $B^n$ the Euclidean unit ball with volume $\kappa_n$. 
For $p,q\in\R^n$ we write $[p,q]$ for the line segment joining the points $p$ and $q$, and $S_v:=[-v,v]$, $v\in\R^n$, for the line segment joining $-v$ and $v$. 

The general linear group in $\R^n$ is denoted by $\GL(n)$, the special linear group by $\SL(n)$,  the group of orthogonal transformations of $\R^n$ by $\O(n)$ and by $\SO(n)\subset\O(n)$ the group of the orthogonal transformations which preserve orientation.

\subsection{Convex bodies}
For the basics on convex geometry and on the theory of valuations, we refer the reader to the books \cite{gardner.book06,gruber.book,klain.rota,schneider.book14,artstein.giannopoulos.milman}.

Let $\K^n$ denote the set of convex bodies (compact and
convex sets) in $\R^n$ endowed with the Hausdorff metric, and let $\K^n_s$ denote the set of convex bodies in $\R^n$ which are symmetric with respect to the origin. The elements of $\K_s^n$ are called \emph{$o$-symmetric} convex bodies.
We endow $\K^n$ with the Minkowski addition:
$$K+L:=\{x+y\,:\,x\in K,\,y\in L\}.$$
The \emph{support function} of a convex body $K\in\K^n$, $h_K:\R^n\longrightarrow\R$, is given by 
$$h(K, u) = h_K(u)=\max \{ \langle u, x\rangle : x \in K \},\quad u\in\R^n,$$
and it determines $K$ uniquely (\cite[Theorem 1.7.1]{schneider.book14}). 
For every $u\in\R^n$, the function $K\mapsto h(K,u)$ is linear with respect to the Minkowski addition and multiplication by non-negative reals:
\begin{equation}\label{add_sup}
h(\alpha K+\beta L,\cdot)=\alpha h(K,\cdot)+\beta h(L,\cdot),\quad\forall\, K,L\in\K^n,\, \forall\, \alpha,\beta\ge0,
\end{equation}

A \emph{zonotope} is a convex body obtained as the finite sum of line segments and a \emph{zonoid} is a convex body that can be approximated, in the
Hausdorff metric, by zonotopes (see e.g. \cite[p.~191]{schneider.book14}).
Zonotopes will play a prominent role in the proof of Theorem~\ref{teo}.

\subsection{Mixed volumes}\label{sec_mixed_volumes}
We will thoroughly use the notion of mixed volumes of convex bodies, for which we refer to Chapter 5 of \cite{schneider.book14}. 
The mixed volume of $n$ convex bodies $K_1,\dots,K_n$ from $\K^n$ will be denoted by the usual notation:
$$
V(K_1,\dots,K_n).
$$ 
Mixed volumes are multilinear functionals $(\K^n)^n\longrightarrow\R$. In each entry, they are continuous, translation invariant, and satisfy the valuation property (see \eqref{val_prel} for the definition). 

Brackets $[i]$ next to an entry of a mixed volume mean that the entry is repeated $i$ times.

Mixed volumes can be extended to the vector space spanned by restrictions of support functions on $\sfe$ (see \cite[p.~291]{schneider.book14}).
For the proof of Theorem~\ref{teo}, we will use the existence of this extension.
In view of this, we will use both notations, $K$ and $h_K$, as arguments in a mixed volume involving the convex body $K$. In other words, we write equivalently 
$$
V(K,K_2,\dots,K_n)\quad\mbox{or}\quad V(h_K,K_2,\dots,K_n)
$$ 
and interpret the support function as a function restricted to $\sfe$.

Along the paper, and especially in Section 3, we will use the following result, containing conditions for which a mixed volume does not vanish. 

\begin{theorem}[Theorem 5.1.8 in \cite{schneider.book14}]\label{mix_volumes_Schneider}
For $K_1,\dots,K_n\in\K^n$, the following assertions are equivalent:
\begin{enumerate}
\item[\emph{(a)}] $V(K_1,\dots,K_n)>0;$
\item[\emph{(b)}] there are segments $S_i\subset K_i$ $(i=1,\dots,n)$ having linearly independent directions;
\item[\emph{(c)}] $\dim(K_{i_1}+\dots+K_{i_k})\geq k$ for each choice of indices $1\leq i_1<\dots<i_k\leq n$  and for all $k\in\{1,\dots,n\}$.
\end{enumerate}
\end{theorem}

\subsection{Valuations}
Let $(\A,+)$ be an Abelian semigroup. A map $\gv:\K^n\longrightarrow(\A,+)$ is called  \emph{valuation} if 
\begin{equation}\label{val_prel}
\gv(K)+\gv(L)=\gv(K\cup L)+\gv(K\cap L),
\end{equation}
for every $K,L\in\K^n$ such that $K\cup L\in\K^n$.

We say that $\gv$ is \emph{translation invariant} if $\gv(K+t)=\gv(K)$ for all $K\in\K^n$ and $t\in\R^n$. If $\A$ is a topological space, we say that $\gv:\K^n\longrightarrow\A$ is \emph{continuous}, if it is continuous with respect to the Hausdorff topology on $\K^n$. 
If there is a multiplication between the positive real numbers and the elements in $\A$, then we say that $\gv:\K^n\longrightarrow\A$ is \emph{homogeneous of degree} $j$, if $\gv(\lambda K)=\lambda^{j}\gv(K)$ for all $\lambda\in(0,\infty)$. If $\A$ is ordered, then $\gv$ is \emph{monotonic} (increasing with respect to set inclusion) if for all $K,L\in\K^n$ such that $K\subset L$, then $\gv(K)\leq\gv(L)$. A valuation is called \emph{even} if $\gv(-K)=\gv(K)$ for all $K\in\K^n$. If $(-1)\cdot\gv(K)=:-\gv(K)$ is defined for all $K$, then we say that $\gv$ is \emph{odd} if $\gv(-K)=-\gv(K)$ for every $K\in\K^n$, and $\gv$ is called an $\emph{$o$-symmetrization}$ if $\gv(K)\in\K^n_s$ for every $K\in\K^n$, that is,
\[
\gv(K)=-\gv(K) \text{ for all } K\in\K^n.
\] 
Finally, if a group of transformations $G$ acts on $\K^n$ and on $\A$, we say that a valuation $\gv:\K^n\longrightarrow\A$ is \emph{$G$-covariant} if, for any $K\in\K^n$, 
\[
\di(g K)=g\di(K) \text{ for all } g\in \,G.
\]

\subsubsection{Real-valued valuations}
These are the valuations $\rv$ on $\K^n$ having $(\R,+)$, the real numbers with the usual addition, as target space. We denote by 
$\Val$ the space of real-valued valuations, which are additionally continuous and translation invariant; this is in fact a Banach space. The subspace of valuations 
homogeneous of degree $j$, $j\in\{0,\dots,n\}$, is denoted by $\Val_j$. 

McMullen proved the following fundamental decomposition result of the space $\Val$.
\begin{theorem}[\cite{mcmullen77}]\label{mcmullen_dec:teo} For every $\rv\in\Val$ there exist unique $\rv_j$,
$j=0,\dots,n$, with $\rv_j\in\Val_j$ for every $j$, such that
$$
\rv=\sum_{j=0}^n \rv_j.
$$
In other words
$$
\Val=\bigoplus_{j=0,\ldots,n} \Val_j.
$$
\end{theorem}

\medskip
The next result provides useful information on the image of a homogeneous valuation in $\Val_j$ which vanishes on convex bodies of certain dimensions.

\begin{theorem}[\cite{klain00,schneider_schuster06}] \label{th_lemma}
Let $\mu\in\Val_j$, $j\in\{0,1,\dots,n-1\}$.
\begin{enumerate}
\item[\emph{(i)}]
If $\mu(K)=0$ for every $K\in\K^n$ with $\dim K= j$, then $\mu(-K)=-\mu(K)$ for every $K\in\K^n$. In particular, $\mu$ is odd.
\item[\emph{(ii)}] If $\mu(K)=0$ for every $K\in\K^n$ with $\dim K= j+1$, then $\mu\equiv 0$.
\end{enumerate}
\end{theorem}

\subsubsection{Minkowski valuations}
An operator $\di:\K^n\longrightarrow\K^n$ is called a \emph{Minkowski valuation} if \eqref{val_prel} holds for $\di$ and $(\mathcal{A},+)=(\K^n,+)$ with $+$ the Minkowski addition of convex bodies.  

The space of continuous and translation invariant Minkowski valuations is denoted by $\MVal$. By $\MVal_j\subset\MVal$ (resp.~$\MVal_j^s\subset\MVal^s$), we denote the $j$-homogeneous Minkowski valuations (resp.~that are $o$-symmetrizations). 

We will often use the following construction to pass from Minkowski valuations to real-valued valuations: Let $\di$ be a Minkowski valuation and fix $u\in\R^n$. The map $\di_u:\K^n\longrightarrow\R$ defined by
$$\di_u(K)=h(\di(K),u),\quad\forall K\in\K^n,$$
is a real-valued valuation which inherits the properties of $\di$ such as continuity, translation invariance, $j$-homogeneity, and monotonicity.

Let $\di:\K^n\longrightarrow\K^n$ be a continuous and translation invariant Minkowski valuation, i.e., $\di\in\MVal$. Using the support function of $\di(K)$ as just described, the decomposition in Theorem~\ref{mcmullen_dec:teo} yields
\begin{equation}\label{h-decomp1}
h(\di (K),u)=\sum_{j=0}^{n}f_{j}(K,u),\quad u\in\R^n,
\end{equation}
where every $f_{j}(K,u)$ is continuous in both variables and positively homogeneous of bi-degree $(j,1)$, i.e., it satisfies
$$f_{j}(\lambda K, \mu u)=\lambda^{j} \mu f_{j}(K,u),\quad\forall\lambda,\mu>0.$$
Notice that, by the McMullen decomposition, each $f_j$ has the valuation property with respect to $K$, for every fixed $u$. 
Moreover, $K\mapsto f_j(K,u)$ is translation invariant for every $u\in\R^n$. 
Since we will use the above decomposition very often, we will refer to it as the {\it McMullen decomposition} of $\di\in\MVal$ instead of McMullen decomposition of $h(\di(\cdot),u)\in\Val$. We would like to remark that in the literature the term ``McMullen decomposition of a Minkowski valuation'' has been used with a stronger meaning, namely,
a Minkowski valuation $\di$ has a McMullen decomposition if there exist $\di_j\in\MVal_j$, $0\leq j\leq n$, such that 
$$\di=\sum_{j=0}^n\di_j.$$ This turns out to be equivalent to the fact that every function $u\mapsto f_j(K,u)$ in \eqref{h-decomp1} is convex for every $0\leq j\leq n$ (cf.~\cite{dorrek}) which is, in general, not the case. This was first shown in \cite{parapatits.wannerer} (see also \cite{dorrek}).

The following two results give conditions in order that some of the functions $f_j(K,u)$ are support functions.

\begin{lemma}[\cite{schneider_schuster06}]\label{support}Let $\di\in\MVal$. If a convex body $K\in\K^n$ satisfies
$$h(\di(\lambda K),\cdot)=\sum_{j=k}^lf_{j}(\lambda K,\cdot),$$
for $\lambda>0$, with some $k,l\in\{0,\dots,n\}$, $k\leq l$, then $f_{k}(K,\cdot)$ and $f_{l}(K,\cdot)$ are support functions.  
\end{lemma}
By Lemma~\ref{support} the functions $f_0(K,\cdot)$ and $f_n(K,\cdot)$ in the McMullen decomposition~\eqref{h-decomp1} are always support functions. Moreover, since  for every $u\in\R^n$, the function $f_0(\cdot,u)$ is a continuous, translation invariant, and homogeneous of degree $0$ real-valued valuation, it is a multiple of the Euler characteristic and, hence, independent of the convex body $K$; notice, however, that this multiple may depend on $u$.  Analogously, $f_n(K,u)$ is a multiple of the volume of $K$ (see \cite{hadwiger}), which may depend on $u$. In the following, we denote by $L_0$ (resp.~$L_n$) the convex body with support function $f_0(\{0\},\cdot)$ (resp.\! $f_n(\kappa_n^{-1/n}B_n,\cdot)$ 
) and write the McMullen decomposition of $\di$ as
\begin{equation}\label{h-decomp}
h(\di(K),u)=h(L_0,u)+\sum_{j=1}^{n-1}f_j(K,u)+\vol(K)h(L_n,u).
\end{equation}

\begin{remark}\label{image of a point} Let $\di\in\MVal$ and let $p\in\R^n$ be a point. Then 
$
\di(\{p\})=L_0.
$
\end{remark}

Another particular case, where the functions $u\mapsto f_j(K,u)$ are known to be convex, was given in \cite{schuster.parapatits}. Parapatits and Schuster proved there that restricted to zonoids $Z\in\K^n$, each function $u\mapsto f_j(Z,u)$ in the McMullen decomposition for $\di\in\MVal$ 
is a support function. 
 
\begin{theorem}[\cite{schuster.parapatits}]\label{suppZonoid} Let $\di\in\MVal$ and let $Z\in\K^n$ be a zonoid. Then there exist convex bodies $L_0$, $\di_1(Z),\dots,\di_{n-1}(Z), L_n$ such that
\begin{equation}\label{h-decomp convex bodies}
\di(\lambda Z)=L_0+\lambda\di_1(Z)+\dots+\lambda^{n-1}\di_{n-1}(Z)+\lambda^n\vol(Z)L_n,
\end{equation}
for every $\lambda>0$.
\end{theorem}

In view of the previous result we fix the following notation. 
\begin{definition}\label{notation_fi_dii}
Let $\di\in\MVal$, let $u\in\R^n$, and let $Z\in\K^n$ be a zonoid. 
\begin{enumerate}[(i)]
\item For $j\in\{0,1,\dots,n\}$, the function $f_j(\cdot,u):\K^n\longrightarrow\R$ will be called the \emph{$j$-homogeneous function of the McMullen decomposition of $\di$}, in \eqref{h-decomp1}. 
\item For $j\in\{1,2,\dots,n-1\}$, $\di_j(Z)\in\K^n$ will be referred to as the \emph{convex body $\di_j(Z)$ of the McMullen decomposition of $\di$} in \eqref{h-decomp convex bodies}.

\noindent To simplify the notation in this case, we also write $\di_0(Z)$ for $L_0$ and 
$\di_n(Z)$ for $\vol(Z)L_n$.
\end{enumerate}

\end{definition}

For every $1\leq j\leq n-1$, the function $u\mapsto f_j(K,u)$ of the McMullen decomposition of $\di\in\MVal$ defined in \eqref{h-decomp1} inherits many invariance properties of $\di$. 
In particular, we easily deduce the following.

\begin{lemma}\label{even}Let $n\geq 2$, let $j\in\{0,\dots,n\}$, and let $\di\in\MVal^s$. Then the $j$-homogeneous function of the McMullen decomposition of $\di$, $u\mapsto f_j(K,u)$, is even for every $K\in\K^n$.
\end{lemma}
\begin{proof}
Let $K\in\K^n$ and let $\di\in\MVal^s$. Since $\di(K)\in\K^n_s$, for every $K\in\K^n$, we have $h(\di(K),u)=h(\di(K),-u)$ for every $u\in\R^n$. 
For $\lambda>0$, the McMullen decomposition for $\di$ in \eqref{h-decomp1} with $\lambda K$ instead of $K$ and once with $-u$ instead of $u$ yields
$$
f_{0}(K,u)+\sum_{j=1}^n\lambda^{j}f_j(K,u)=f_{0}(K,-u)+\sum_{j=1}^n\lambda^{j}f_j(K,-u).
$$
By comparing the coefficients of the above polynomial expression in $\lambda$ we get $f_{j}(K,u)=f_{j}(K,-u)$ for every $0\leq j\leq n$.
\end{proof}

In the next lemma we collect some facts about the functions involved in the McMullen decomposition of $\di\in\MVal$ which will be used throughout the rest of the work. 

\begin{lemma}\label{r: facts on f_js}
Let $\di\in\MVal$, $K\in\K^n$, $u\in\R^n$, and $j\in\{0,1,\dots,n\}$. Then:
\begin{enumerate}
\item[\emph{(i)}] the function $u\mapsto f_j(K,u)$ is a difference of support functions of convex bodies;
\item[\emph{(ii)}] the function $K\mapsto f_j(K,u)$ is a valuation homogeneous of degree $j$; 
\item[\emph{(iii)}] if $\dim K\leq j-1$, then $f_j(K,u)=0$;
\item[\emph{(iv)}] if $\dim K=j$, then $u\mapsto f_j(K,u)$ is a support function;
\item[\emph{(v)}] if $f_j(K,u)=0$ for every $K\in\K^n$ with $\dim K=j+1$, then $f_j(\cdot,u)\equiv 0$.
\item[\emph{(vi)}] if $j_0\in\{0,\dots,n-1\}$ and $f_j(K,\cdot)$ is linear for every $j>j_0$, then $f_{j_0}(K,\cdot)$ is a support function. 
\end{enumerate}
\end{lemma}
\begin{proof}
The statement of item (i) was proved in \cite{schuster.parapatits}. Item (ii) follows directly from the McMullen decomposition of $\di$. Item (iii) follows, for instance, from Corollary~6.3.2 in \cite{schneider.book14}. 
Item (iv) is deduced from Lemma~\ref{support} and items (ii) and (iii). Item (v) follows from Theorem~\ref{th_lemma}(ii).

For item (vi), we first note that since $\di(K)$ is a convex body for every $K\in\K^n$, we have
$0\geq h(\di(\lambda K),u+v)-h(\di(\lambda K),u)-h(\di(\lambda K),v)$ for every $u,v\in\R^n$ and $\lambda>0$. By the McMullen decomposition of $\di$ and the linearity of $f_j(K,\cdot)$ for $j>j_0$, we have 
\begin{align*}
0&\geq \lambda^n(f_n(K,u+v)-f_n(K,u)-f_n(K,v))+\dots\\
&\quad\dots+\lambda^{j_0+1}(f_{j_0+1}(K,u+v)-f_{j_0+1}(K,u)-f_{j_0+1}(K,v))+
\\&\quad+\lambda^{j_0}(f_{j_0}(K,u+v)-f_{j_0}(K,u)-f_{j_0}(K,v))
+O(\lambda^{j_0-1})
\\&=\lambda^{j_0}(f_{j_0}(K,u+v)-f_{j_0}(K,u)-f_{j_0}(K,v))
+O(\lambda^{j_0-1}).
\end{align*}

If $j_0\geq 1$, then as $\lambda\to\infty$, we get that the inequality can be satisfied only if 
$$f_{j_0}(K,u+v)-f_{j_0}(K,u)-f_{j_0}(K,v)\leq 0,$$ that is, 
$f_{j_0}(K,\cdot)$ is a support function for every $K\in\K^n$. If $j_0=0$, we obtain the latter directly.
\end{proof}

\medskip
To finish this section, we state the following technical result, which can be obtained from Theorem~6.3.6 in \cite{schneider.book14}. For completeness, we give a proof of the result, which will be essential in Section~\ref{sec: dependence a point}.

Let $C(n,k)$ denote the set of all ordered subsets of $k$ elements among $\{1,\dots,n\}$ and let $\sigma_j$ be the $j$-th element of $\sigma$, $1\leq j\leq k$.
 
\begin{theorem}[Corollary of Theorem 6.3.6 in \cite{schneider.book14}]
\label{lem_polarization}
Let $1\leq k\leq n$, let $\di\in\MVal_k$, and let $S_1,\dots,S_n$ be segments in $\R^n$. Then
$$\di(S_1+\dots+S_n)=\sum_{\sigma\in C(n,k)}\di(S_{\sigma_1}+\dots+S_{\sigma_k}).$$ 
\end{theorem}
\begin{proof}
Let $\di\in\MVal_k$ and let $S_1,\dots,S_n$ be segments in $\R^n$. 
Consider $u\in\R^n$ and define the continuous and translation invariant real-valued valuation $\Phi_u(K):=h(\Phi(K),u)$. From Theorem 6.3.6 in \cite{schneider.book14}, we have that there exists a continuous and translation invariant operator $\overline{\di_u}:(\K^n)^k\longrightarrow\R$ that is Minkowski additive in each variable and such that 
\begin{equation}\label{636_schneider_u}
\di_u(S_1+\dots+S_n)=\sum_{r_1,\dots,r_n=0}^{k}\binom{k}{r_1\dots r_n}\overline{\di_u}(S_1[r_1],\dots,S_n[r_n]),
\end{equation}
with $\sum_{j=1}^nr_j=k$. 

Moreover, by Theorem 6.3.6 in \cite{schneider.book14}, the mapping $K\mapsto{\overline{\di_u}}(K[r],M_{r+1},\dots,M_k)$ is a continuous and translation invariant valuation, homogeneous of degree $r$ for each fixed $r\in\{1,\dots,k\}$ and for every fixed tuple of convex bodies $M_{r+1},\dots,M_k$. In particular, for every $r_1,\dots,r_n$ with $r_1+\dots+r_n=k$, we have that 
$$K\mapsto{\overline{\di_u}}(K[r_1],S_2[r_2],\dots,S_n[r_n])$$
is a continuous and translation invariant real-valued valuation, homogeneous of degree $r_1$. Hence, if $\dim K<r_1$, then $\overline{\di_u}(K[r_1],S_2[r_2],\dots,S_n[r_n])=0$ (see \cite[Corollary 6.3.2]{schneider.book14}). Since in \eqref{636_schneider_u} we are taking $K=S_1$, a segment, if $r_1\geq 2$, the summand vanishes. Since the same argument can be done for $r_2,\dots,r_n$, we obtain that $r_i\in\{0,1\}$ for every $1\leq i\leq n$ and the sum in \eqref{636_schneider_u} can be taken over $C(n,k)$. Hence,  
\begin{align*}
\di_u(S_1+\dots+S_n)&=\sum_{r_1,\dots,r_n=0}^{k}\binom{k}{r_1\dots r_n}\overline{\di_u}(S_1[r_1],\dots,S_n[r_n])
\\&=\sum_{\sigma\in C(n,k)}\overline{\di_u}(S_{\sigma_1},\dots,S_{\sigma_k})
=\sum_{\sigma\in C(n,k)}\di_u(S_{\sigma_1}+\dots+S_{\sigma_k}).
\end{align*}
The last equality holds by applying the same argument as before but with $\di_u(S_{\sigma_1}+\dots+S_{\sigma_k})$ instead of $\di_u(S_1+\dots+S_n)$.
Since $\di_u(K)=h(\di(K), u)$, by using \eqref{add_sup}, we have proven that 
\[
h(\Phi(S_1+\dots+S_n),u)=h(\sum_{\sigma\in C(n,k)}\Phi(S_{\sigma_1}+\dots+S_{\sigma_k}),u),
\]
for every $u\in\R^n$. 
Since the support function uniquely describes a convex body (\cite[Theorem 1.7.1]{schneider.book14}), the statement of the theorem follows.  
\end{proof}

\subsection{Volume constraints}\label{sec:volume_constraints}
As described in the introduction, the main objective of this paper is to describe Minkowski valuations satisfying certain volume constraints. 
\begin{definition} An operator $\di:\K^n\longrightarrow\K^n$ satisfies a lower volume constraint (LVC) if there exists a constant $c_\di>0$ such that
$$\vol(\di(K))\ge c_\di\vol(K),\quad\forall K\in\K^n.$$
Analogously, we say that $\di$ satisfies an upper volume constraint (UVC) if there exists $C_\di>0$ such that
$$\vol(\di(K))\leq C_{\di}\vol(K),\quad\forall K\in\K^n.$$
\end{definition}

Throughout the paper we will refer to these properties simply writing (LVC) and (UVC), respectively. We will mostly consider
valuations that satisfy both (LVC) and (UVC), which corresponds to Definition \ref{def_VC_intro}. If $\di$ is of this type, we will say that $\di$ satisfies the {\em volume constraint}, briefly, $\di$ satisfies (VC) or $\di$ satisfies the (VC) condition. 

The identity operator on $\K^n$ trivially satisfies (VC), but a more interesting example, which  motivated the previous definition in \cite{abardia.colesanti.saorin1}, 
is the difference body operator, defined in \eqref{diff_body}, which satisfies (RS). 

The operators in Theorem~\ref{teo}(ii) are also examples of Minkowski valuations satisfying (VC). Indeed, for a segment $S$ and an $(n-1)$-dimensional convex body $L$ with $\dim(L+S)=n$ we have, by the linearity and positivity of mixed volumes (see Theorem~\ref{mix_volumes_Schneider}),
$$\vol(L+\vol(K)S)=V(L[n-1],\vol(K)S)=\vol(K)V(L[n-1],S)=\vol(K)\vol(L+S).$$
Hence, the (VC) condition is satisfied with $c_{\di}=C_{\di}=\vol(L+S)\neq 0$.

\section{Dichotomy for the image of a point}\label{sec: dependence a point}

The aim of this section is to prove that, if a Minkowski valuation $\di\in\MVal$ satisfies the (VC) condition, then the image of a 
point is either a point or an $(n-1)$-dimensional convex body. That is, we prove the following result.

\begin{theorem}\label{L0Ln} Let $p\in\R^n$. 
If $n\geq 2$ and $\di\in\MVal$ satisfies the (VC) condition, then either $\dim(\di(\{p\}))=0$ or $\dim(\di(\{p\}))=n-1$.
\end{theorem}  

For the proof of Theorem~\ref{L0Ln}, we need to exploit in a more specific way the information given by the McMullen decomposition \eqref{h-decomp} of $\di$, which we
can use since $\di$ is a continuous and  translation invariant Minkowski valuation.

We consider $\vol(\di(\lambda K))$ for $\lambda>0$. By \eqref{h-decomp}, Lemma~\ref{r: facts on f_js}(i), and the extension of mixed volumes to differences of support 
functions (see Section~\ref{sec_mixed_volumes}), we have 
$$
\vol(\di (\lambda K))\!=\!V_n(h(\di (\lambda K),\cdot)[n])\!=\!V_n\Big(\!\big(h(L_0,\cdot)+\sum_{j=1}^{n-1}\lambda^j f_j(K,\cdot)+\lambda^n\vol(K)h(L_n,\cdot)\big)[n]\Big).
$$
The multilinearity of the extension of mixed volumes to differences of support functions provides us with a polynomial expansion of $\vol(\di(\lambda K))$ in $\lambda$, which may contain terms of degree from 
$0$ until $n^n$.
Moreover, each of the coefficients of the polynomial is a sum of mixed volumes of the support functions of $L_0$ and $L_n$, and the 
functions $f_j(K,\cdot)$, $1\leq j\leq n-1$, involved in the McMullen decomposition of $\di$.
As each of these functions depends only on $K$, for the sake of brevity, we will write 
\begin{equation}\label{vi di}
\vol(\di(\lambda K))=\sum_{j=0}^{n^n} v^{\di}_j(K) \lambda^j,
\end{equation}
and denote by $v_j^{\di}(K)$ the coefficient of degree $j$ in the above polynomial expansion of $\vol(\di(\lambda K))$, $0\leq j \leq n^n$. 
We note that, in general, $v_j^{\di}(K)$ may be negative. 

\begin{remark}\label{f_i in v_j}
Let $\di\in\MVal$ satisfy (VC) and let $f_i$, $0\leq i\leq n$, be the functions appearing in its McMullen decomposition. Let $K\in\K^n$ and let $v_j^{\di}(K)$ be as above. Then, for every $0\leq j \leq n^n$, the mixed volumes involved in the coefficient $v_j^{\di}(K)$ contain only $f_i(K,\cdot)$ with $0\leq i\leq j$.
\end{remark}

We next state a fact whose proof is a simple observation, but which will play an important role in the next.

If $\di\in\MVal$ satisfies the (VC) condition, then there exist positive constants $c_\di$ and $C_\di$, 
independent of $K$ and $\lambda$, for which
$$
c_\di\lambda^n\vol(K)=c_\di\vol(\lambda K)\leq\vol(\di(\lambda K))\leq C_\di\vol(\lambda K)=C_\di\lambda^n\vol(K). 
$$
Comparing these inequalities with \eqref{vi di}, we immediately get that the only possibly non vanishing term in the sum in~\eqref{vi di} is the one containing $\lambda^n$. In other words,   
$\vol(\di(\lambda K))$ is necessarily a monomial of degree $n$. 
The following corollaries collect the important consequences of this fact. 

\begin{corollary}\label{mixedZero}
Let $\di\in \MVal$ satisfy the (VC) condition. 
Then: 
\begin{enumerate}
\item[\emph{(i)}] if $\dim K<n$, then $v^{\di}_l(K)=0$ for all $0\leq l \leq n^n$;
\item[\emph{(ii)}] if $\dim K=n$, then $v_n^{\di}(K)\neq 0$;
\item[\emph{(iii)}] if $\dim K=n$, then $v^{\di}_l(K)=0$ for every $l\neq n$.
\end{enumerate}
\end{corollary}

\begin{corollary}\label{each_zero}
Let $\di\in \MVal$ satisfy the (VC) condition and let $K\in\K^n$ be fixed.
If for every $1\leq j\leq n-1$, the functions $u\mapsto f_j(K,u)$ are convex,
then a coefficient $v^{\di}_j(K)$ in \eqref{vi di} vanishes if and only if
each of the mixed volumes involved in its explicit expression does.
\end{corollary}

\begin{proof}
Let $\di\in \MVal$ satisfy the (VC) condition and let $K\in\K^n$.
If each function $u\mapsto f_j(K,u)$ is a convex function, $1\leq j\leq n-1$, then the coefficients $v^{\di}_j(K)$ are sums of mixed volumes of convex bodies. Thus, these summands are all non-negative and the statement holds.
\end{proof}

We now proceed to prove Theorem~\ref{L0Ln}. 

\begin{proof}[Proof of Theorem~\ref{L0Ln}] 
Let $\di\in\MVal$ satisfy the (VC) condition. We consider its McMullen decomposition, as described in \eqref{h-decomp}, and use the notation of Definition~\ref{notation_fi_dii}.

We first prove that $\dim L_0\neq n$. 
Indeed, since $L_0=\di(\{p\})$, by Remark~\ref{image of a point} and \eqref{h-decomp}, we have that $\dim L_0=n$ implies $\vol(\di(\{p\}))>0$, in contradiction with the (UVC) condition. From now on, we assume $0\leq \dim L_0\leq n-1$. 

Let $Z$ be a fixed $n$-dimensional zonotope. By Theorem~\ref{suppZonoid}, $u\mapsto f_k(Z,u)$ is the support function of a convex body $\di_k(Z)$,
for every $1\leq k \leq n-1$. 
On the other hand, by Corollary~\ref{mixedZero}(ii) we have $v_n^{\di}(Z)\neq 0$, where $v_n^{\di}(Z)$ is the coefficient of the degree $n$ of the polynomial $\vol(\di(\lambda Z))$, in $\lambda$, given in \eqref{vi di}. Therefore, $v_n^{\di}(Z)$ is a 
sum of mixed volumes of the form
$$
V(\di_0(Z)[a_0],\di_1(Z)[a_1],\dots,\di_{n-1}(Z)[a_{n-1}],\di_n(Z)[a_n]),
$$
where $a_0,\dots,a_n\in\{0,1,\dots,n\}$ and satisfy the following two conditions.
\renewcommand\labelitemi{\tiny$\bullet$}
\begin{itemize}\itemsep9pt
\item As the sum of the multiplicities of the entries of a mixed volume is $n$, 
\begin{equation}\label{sum=n}
\sum_{k=0}^na_k=n;
\end{equation}   
\item by \eqref{vi di} and the fact that $V(\di_0[a_0],\di_1(Z)[a_1],\dots,\di_{n-1}(Z)[a_{n-1}],\di_n(Z)[a_n])$ is a summand of $v_n^{\di}(Z)$, 
\begin{equation}
\label{sumk=n}\sum_{k=0}^nka_k=n.
\end{equation}
\end{itemize}

Using the described notation, we prove the following claim. 

\noindent{\bf Claim 1.} \emph{$v_n^{\di}(Z)$ has only one non-zero summand:
\begin{equation}\label{gen_term}
V(\di_0(Z)[a_0],\di_1(Z)[a_1],\dots,\di_{n-1}(Z)[a_{n-1}],\di_n(Z)[a_n])>0,
\end{equation} 
with $a_0,\dots,a_n\in\{0,1,\dots,n\}$ satisfying \eqref{sum=n} and \eqref{sumk=n}. Moreover, 
$$
\dim(\di_k(Z))=a_k,\quad\forall\, k\in\{0,\dots,n\}.
$$}
We note that, a priori, $a_k$, $1\leq k\leq n-1$, may depend on the zonotope $Z$.

At the end of the proof, we will show that this summand is one of the following two:
\begin{enumerate}\itemsep6pt
\item[(A)] $V(L_0[n-1],L_n)$,
\item[(B)] $V(\di_1Z[n])$,
\end{enumerate}
and that this fact implies Theorem~\ref{L0Ln}. We note that the mixed volume on (A) (resp. (B)) corresponds to $a_0=n-1$, $a_n=1$ and $a_j=0$, $1\leq j\leq n-1$ (resp. $a_1=n$, $a_0=0$ and $a_j=0$, $2\leq j\leq n$), which are the trivial solutions of \eqref{sum=n} and \eqref{sumk=n}. 

\medskip We next prove Claim~1. By Corollary~\ref{mixedZero}(ii), there exist $a_0,\dots,a_n$ for which \eqref{gen_term} holds. 
We show that $\dim(\di_k(Z))=a_k$ for every $k\in\{0,\dots,n\}$. This means, in particular, that there is only one possible choice for the numbers $(a_0,a_1,\dots,a_n)$, which implies the whole claim. 

For $a_0,\dots,a_n$ such that \eqref{gen_term} holds, Theorem~\ref{mix_volumes_Schneider} yields that 
$\dim(\di_k(Z))\geq a_k$ and, hence, there exist $a_k$ linearly independent segments $S_1,\dots,S_{a_k}\subset\di_k(Z)$, for every $k\in\{0,\dots,n\}$.
Assume that for some $k\in\{0,\dots,n\}$, $\dim(\di_k(Z))>a_k$. Then, by condition (b) in  Theorem~\ref{mix_volumes_Schneider}, there exists $j\ne k$ such that
$a_j\ge 1$ and 
$$
V(\di_0(Z)[a_0],\dots,\di_k(Z)[a_k+1],\dots,\di_j(Z)[a_j-1],\dots,\di_n(Z)[a_n])>0.
$$ 
However, this mixed volume is one of the summands of the coefficient $v_{n+k-j}^{\di}(Z)$, which has to be zero by Corollaries~\ref{each_zero} and \ref{mixedZero}(iii). Thus, $\dim(\di_k(Z))=a_k$, which concludes the proof of Claim~1.

\smallskip
In the second step of the proof, we will apply Claim~1 to cubes. To do so, we need to introduce some notation. 
Let $\{w_1,\dots,w_n\}$ be a fixed basis of $\R^n$ and define 
$S_i:=[-w_i,w_i]$ and 
\begin{equation}\label{Cn}
C_n:=S_1+\dots+S_n.
\end{equation} 
Clearly $C_n$ is an $n$-dimensional zonotope and since $\di$ satisfies (VC) we have $\dim(\di (C_n))=n$.

We define $a_k:=\dim(\di_k(C_n))$. By Claim~1, $a_k$ coincides with the multiplicity of $\di_k(Z)$ in the mixed volume appearing in \eqref{gen_term}, for $Z=C_n$. 

Let $C(n,k)$ denote, as in Theorem \ref{lem_polarization}, the set of all ordered subsets of $k$ elements among $\{1,\dots,n\}$ and for $\sigma\in C(n,k)$ and $j=1,\dots,k$, 
let $\sigma_j$ denote the $j$-th element of $\sigma$. By Theorem~\ref{lem_polarization},  we have
\begin{align}\label{polarization}
\di_k(C_n)&=\sum_{\sigma\in C(n,k)}\di_k(S_{\sigma_1}+\dots+S_{\sigma_k})\\\nonumber
&=\sum_{\sigma\in C'(n,k)}\di_k(S_{\sigma_1}+\dots+S_{\sigma_k})+\sum_{\sigma
\in C(n,k)\setminus C'(n,k)}\di_k(S_{\sigma_1}+\dots+S_{\sigma_k}),
\end{align}
where $C'(n,k)$ contains those elements $\sigma\in C(n,k)$ for which $\dim(\di_k(S_{\sigma_1}+\dots+S_{\sigma_k}))\neq 0$.
For every $1\leq k\leq n$, we can choose a subset $\Sigma_k\subset C'(n,k)$ which is {\em minimal} in the following sense: first
\begin{equation}\label{minimal_elements}
\dim(\di_k(C_n))=\dim\left(\sum_{\sigma\in\Sigma_k}\di_k(S_{\sigma_1}+\dots+S_{\sigma_k})\right),
\end{equation} 
and, secondly, this equality fails to be true if we omit one of the terms from $\Sigma_k$ in the sum on the right-hand side. 
We note that the number of elements in $\Sigma_k$ is at most $a_k$, which is attained if $\dim(\di_k(S_{\sigma_1}+\dots+S_{\sigma_k}))=1$ for every $\sigma\in\Sigma_k$. Moreover, for every $\sigma\in C'(n,k)$ there exists a subset $\Sigma_k$ which contains $\sigma$ and is minimal. Equation \eqref{minimal_elements} will be used to prove Claim~2 below. For simplicity, we say that \emph{a segment $S_j$ has index in $\Sigma_k$} if there is a $\sigma\in\Sigma_k$ such that $\sigma_l=j$ for some $1\leq l\leq k$.

Equation \eqref{polarization} together with the McMullen decomposition \eqref{h-decomp} yields 
\begin{equation}\label{pol2}
\di (C_n)=\di_0(C_n)+\sum_{k=1}^n\di_k(C_n)=L_0+\sum_{k=1}^{n}\sum_{\sigma\in C'(n,k)}\di_k(S_{\sigma_1}+\dots+S_{\sigma_k})
+q,
\end{equation}
where $q\in\R^n$ is given by $\sum_{k=1}^n\sum_{\sigma
\in C(n,k)\setminus C'(n,k)}\di_k(S_{\sigma_1}+\dots+S_{\sigma_k})$.

We will next focus on the following sum of convex bodies:
\begin{equation}\label{pol3}
\sum_{k=1}^{n}\sum_{\sigma\in C'(n,k)}\di_k(S_{\sigma_1}+\dots+S_{\sigma_k}).
\end{equation}

Let $\tau_i$ be the number of subsets $\sigma\in C'(n,k)$, for all possible choices of $k$ between 1 and $n$, for which $i$ is an element of $\sigma$. 
In other words, $\tau_i$ is the number of summands in \eqref{pol3} in which the segment $S_i$ appears. 
We define $I:=(\tau_1,\dots,\tau_n)\in\N^n$. 

\medskip
\noindent{\bf Claim 2.}
\begin{enumerate}
{\em
\item[\emph{(a)}] For every $1\leq k\leq n$ and $\sigma\in C'(n,k)$,
\begin{equation}\label{0or1}
\dim(\di_k(S_{\sigma_1}+\dots+S_{\sigma_k}))=1.
\end{equation}
\item[\emph{(b)}] For every $1\leq k\leq n$, there are exactly $a_k$ elements $\sigma\in C'(n,k)$ for which \eqref{0or1} holds.
\item[\emph{(c)}] $I=(1,\dots,1)$.}
\end{enumerate}

First we prove that 
\begin{equation}\label{tau_i_geq_1}\tau_k\geq 1, \quad 1\leq k\leq n,\end{equation} 
arguing by contradiction. Without loss of generality we assume that $\tau_1=0$, i.e., 
$$
\dim(\di_k(S_1+S_{\sigma_2}+\dots+S_{\sigma_k}))=0,\quad\forall\, \sigma=(1,\sigma_2,\dots,\sigma_k)\in C(n,k),\;\forall\, k\in\{1,\dots,n\}.
$$
Then, by \eqref{polarization}, we clearly have 
$$
0<\vol(\di(S_1+\dots+S_n))=\vol(\di(S_2+\dots+S_n))
$$
which is a contradiction with (UVC) since $\vol(S_2+\dots+S_n)=0$.  Hence, $\tau_k\geq 1$, $1\leq k\leq n$, and each segment appears at least in one summand in 
\eqref{pol3}.

For the proof of (a) and (b) in Claim~2, we will repeatedly use the following argument: if our claim is not satisfied, we construct, according to the given considerations in each case, appropriate zonotopes so that (UVC) fails to hold for them.

\smallskip
We prove next that \eqref{0or1} holds.  If $k=n$, then \eqref{0or1} is directly satisfied. Indeed, if $a_n\neq 0$, then $a_n=1$ from
\eqref{sumk=n} and, by Claim~1, $\dim(\di_n(C_n))=\dim(\di_n(S_1+\dots+S_n))=1$.

For $1\leq k\leq n-1$, we prove \eqref{0or1} by contradiction. 
Assume that there are $k\in\{1,\dots,n-1\}$ and $\widetilde\sigma\in C'(n,k)$ such that
$$
\dim(\di_k(S_{\widetilde\sigma_1}+\dots+S_{\widetilde\sigma_k}))\geq 2.
$$ 
Let $\Sigma_k\subset C'(n,k)$ be a minimal set containing $\widetilde\sigma$, as defined in \eqref{minimal_elements}. In this situation, because of the minimality of $\Sigma_k$ and Claim~1, the number $q$ of elements of $\Sigma_k$ is at most $(a_k-1)$. We denote by $s\leq qk$ the number of linearly independent segments with index in $\Sigma_k$ and by $P_s$ the zonotope sum of these $s$ segments. For every $1\leq l\leq n$, $l\neq k$, let $\Sigma_l$ be a fixed minimal set. Denote by $s'$ the number of linearly independent segment with index in an element of the set $\{\Sigma_{l}\}_{1\leq l\leq n, l\neq k}$. Let $P_{s'}$ be the zonotope sum of these $s'$ segments. We have 
$$
s'\le\sum_{l=1,\dots,n;\;l\neq k}la_l=n-ka_k.
$$
Let $P$ be the zonotope given by $P=P_s+P_{s'}$. By construction 
$$\dim P\leq s+s'\leq qk+n-ka_k\le (a_k-1)k+n-ka_k=n-k<n$$ 
and, hence, $\vol(P)=0$. On the other hand, recalling \eqref{minimal_elements}, we have 
$\vol(\di(P))\neq 0$. This contradicts (UVC). Thus, we have \eqref{0or1}. 

\smallskip
Hence, \eqref{tau_i_geq_1} and Claim~2(a) prove that each segment $S_1,\dots,S_n$ appears at least in one summand of \eqref{pol3} and that each summand of 
\eqref{pol3} has dimension 1. This implies that, in the sum in \eqref{polarization}, there are at least $a_k$ summands in 
$C'(n,k)$ for every $1\leq k\leq n-1$. In particular, any minimal set $\Sigma_k$ contains exactly $a_k$ elements.

\medskip
We show next Claim~2(b), i.e., we show that for every $1\leq k\leq n$ there are exactly $a_k$ subsets $\sigma\in C'(n,k)$ for which $\dim(\di_k(S_{\sigma_1}+\dots+S_{\sigma_k}))=1$ holds. 
For $k=n$, this follows immediately since $C(n,n)$ contains only one element and either $a_n=0$ or $a_n=1$. Claim~1 yields the result. 
We prove the statement for $1\leq k\leq n-1$ arguing by contradiction. For $1\leq l\leq n-1$, let $\Sigma_l$ be minimal and denote 
$\sigma^1,\dots,\sigma^{a_l}\in \Sigma_l$. 
Fix $k\in\{1,\dots,n-1\}$ and assume that Claim~2(b) does not hold for $k$. 
Then there is $\sigma\in C'(n,k)$ such that $\sigma\neq\sigma^{j}$, $j\in\{1,\dots,a_k\}$. 
Without loss of generality, since \eqref{0or1} holds, we may assume that 
\begin{equation}\label{eq_exchange}
\dim(\di_{k}(S_{\sigma_1}+\dots+S_{\sigma_k})+\di_{k}(S_{\sigma^2_1}+\dots+S_{\sigma^2_k})+\dots+\di_{k}(S_{\sigma^{a_k}_1}+\dots+S_{\sigma^{a_k}_k}))=a_k.
\end{equation}
Let $Q$ be the sum of the $n-ka_k$ segments whose index is in $\Sigma_l$, $1\leq l\leq n$, $l\neq k$ (cf. Claim~2(a)). Then $\dim Q=n-k a_k$. Indeed, if these segments are not linearly independent, i.e., if $\dim Q<n-ka_k$, consider the zonotope 
$$
P:=\sum_{j=1}^{a_k}\sum_{i=1}^kS_{\sigma_i^j}+Q.
$$ 
Then $\dim P <n$. On the other hand, by \eqref{minimal_elements}, $\vol(\di(P))\neq 0$. This contradicts the (UVC) condition. 
Thus we assume next that $\dim Q=n-ka_k$. 

Consider the at most $(k+1)a_k$ segments $S_{\sigma_1},\dots,S_{\sigma_k},S_{\sigma_1^1},\dots,S_{\sigma_k^1},\dots,S_{\sigma_1^{a_k}},\dots,S_{\sigma_k^{a_k}}$. We will distinguish the following mutually excluding cases and define an appropriate zonotope $P'$ in each case:
\begin{enumerate}[i)]
\item Some segment with index in $\Sigma_k$ is already a summand of $Q$. We set
$$
P'=\sum_{j=1}^{a_k}\sum_{i=1}^k S_{\sigma_i^j}.
$$ 
\item No segment with index in $\Sigma_k$ is in $Q$ but some segment $S_{\sigma_1},\dots,S_{\sigma_k}$ is already a summand of $Q$. In this case we set 
$$
P'=S_{\sigma_1}+\dots+S_{\sigma_k}+\sum_{j=2}^{a_k}\sum_{i=1}^kS_{\sigma_{i}^j}.
$$ 

\item Otherwise, all segments $S_{\sigma_1},\dots,S_{\sigma_k}$ have index in $\Sigma_k$, that is, 
$$
\dim(S_{\sigma_1}+\dots+S_{\sigma_k}+\sum_{j=1}^{a_k}\sum_{i=1}^kS_{\sigma_{i}^j})=ka_k.
$$ 
Since $\sigma\neq \sigma^1$, there is an $l$ such that $2\leq l\leq a_k$ and for which  $S_{\sigma_i}=S_{\sigma^l_r}$ for some $1\leq i,r\leq k$. We set 
$$
P'=S_{\sigma_1}+\dots+S_{\sigma_k}+\sum_{l=2}^{a_k}\sum_{i=1}^kS_{\sigma_{i}^l}.
$$ 
\end{enumerate}

Define the zonotope $P:=P'+Q$. By construction, we have $\dim P<n$. On the other hand, by \eqref{pol2}, \eqref{minimal_elements}, and \eqref{eq_exchange}, 
$\dim(\di(P))=n$. This contradicts the (UVC) condition and Claim~2(b) holds also for $1\leq k\leq n-1$. 

\smallskip
Now, the assertion $I=(1,\dots,1)$, which completes the proof of Claim~2, follows immediately from \eqref{sumk=n}, Claim~2(a) and~(b), and the 
fact that $\tau_i\geq1$, $1\leq i\leq n$. Indeed, by Claim~2(b), $C'(n,k)$ contains exactly  $a_k$ elements, for every $k=1,\dots,n$. This means that
there are exactly $ka_k$ indices corresponding to $C'(n,k)$ and, in total, we have $\sum_{k=1}^n ka_k=n$ indices. As each $\tau_i$ is at least 1, we have that 
$\tau_i=1$ for every $i=1,\dots,n$, and Claim~2 is proved.

\medskip
In the next claim we study the relation between the subset $C'(n,k)$ associated to a generalized cube $C_n$ and the subset $C'(n,k)$ associated to another generalized cube, $P$ that differs only in one segment with $C_n$; that is, we compare the distribution of the segments appearing in \eqref{pol3} within the different $\di_k$ for the generalized cubes $C_n$ and $P$. 

\smallskip
\noindent\textbf{Claim 3.}
\emph{Let $C_n=S_1+\dots+S_n$ be as before and let $S$ be a segment such that $\spa S\neq \spa S_i$ for every $1\leq i\leq n$. Define $P:=S+S_2+\dots+S_n$. Let $j\in\{1,\dots,n-1\}$. If 
$$\dim(\di_{j}(S_1+\dots+S_{j}))=1,$$
then }
\begin{enumerate}\itemsep6pt
\item[(a)] $\dim(\di_j(S+S_2+\dots+S_j))=1$ \emph{and} 
\item[(b)] \emph{$\dim(\di_k(S+S_{\sigma_2}+\dots+S_{\sigma_k}))=0$ for every $k\in\{1,\dots,n\}$ and $S_{\sigma_2},\dots,S_{\sigma_k}$ such that $\{2,\dots,j\}\neq\{{\sigma_2},\dots,{\sigma_k}\}$.}
\end{enumerate}

\smallskip
Let $C_n=S_1+\cdots+S_n$ and let $P=S+S_2+\cdots+S_n$ be as in the statement. 
Using \eqref{pol2} for $C_n$ and $P$, we can write
\[
\Phi(C_n)=L_0+\sum_{k=1}^n\sum_{\sigma\in C'(n,k,C_n)}\Phi_k(S_{\sigma_1}+\dots+S_{\sigma_k}) + q, \quad\text{  and}
\]
\[
\Phi(P)=L_0+\sum_{k=1}^n\sum_{\sigma\in C'(n,k,P)}\Phi_k(S_{\sigma_1}+\dots+S_{\sigma_k}) + q',
\]
where $C'(n,k,C_n)$ (resp. $C'(n,k,P)$) denotes the subset of elements $\sigma\in C(n,k)$ for which $\dim(\Phi_k(S_{\sigma_1}+\cdots+S_{\sigma_k}))\neq 0$, for the above sum in $C_n$ (resp. $P$). For $P$, we make an abuse of notation and denote also by $1$ the index associated to $S$.

We compare the central sum
$$A=\sum_{k=1}^n\sum_{\sigma\in C'(n,k,C_n)}\Phi_k(S_{\sigma_1}+\dots+S_{\sigma_k})$$
 with the sum
$$B=\sum_{k=1}^n\sum_{\sigma\in C'(n,k,P)}\Phi_k(S_{\sigma_1}+\dots+S_{\sigma_k}).$$

First, using Claim~2(c), we split $A$ and $B$ as follows:
$$A=\Phi_j(S_1+\cdots+S_j)+\sum_{k=1}^n\sum_{\substack{\sigma\in C'(n,k,C_n) \\ \sigma_l\neq 1, 1\leq l \leq k}}\Phi_k(S_{\sigma_1}+\dots+S_{\sigma_k})=\Phi_j(S_1+\cdots+S_n)+C,$$ 
$$B=\Phi_i(S+S_{\beta_2}+\cdots+S_{\beta_i})+\sum_{k=1}^n\sum_{\substack{\gamma\in C'(n,k,P) \\ \gamma_l\neq 1, 1\leq l \leq k}}\Phi_k(S_{\gamma_1}+\dots+S_{\gamma_k})=\Phi_i(S+S_{\beta_2}\cdots+S_{\beta_i})+D,$$
for $1\leq j,i\leq n$ and $\beta_2,\dots,\beta_i\in\{2,\dots,n\}$.

By Claim~2(c), $S_1$ does not appear in $C$, as well as, $S$ does not appear in $D$. Thus, every summand in $C$ is a summand in $D$ and vice versa, that is, the sums $C$ and $D$ are the same and contain the same segments. 
Therefore, using again Claim~2(c), we obtain that $\beta_m\in\{2,\cdots,j\}$ for $2\leq m\leq i$. Since every segment $S_m$, $2\leq m\leq j$ appears exactly once in $B$, we necessarily have $i=j$. Hence, the proof of (a) is completed. Now (b) follows directly from (a) and Claim~2(c). 

\medskip
\noindent\textbf{Claim 4.} \emph{Let $C_n=S_1+\dots+S_n$ be as in \eqref{Cn}. Then the unique non-zero summand of $v_n^{\di}(C_n)$ (given by Claim~1) is necessarily one of the following:
\begin{enumerate}\itemsep6pt
\item[\emph{(A)}] $V(L_0[n-1],L_n)\vol(C_n)$,
\item[\emph{(B)}] $V(\di_1(C_n)[n])$.
\end{enumerate}}

\smallskip
In the notation of Claim~1, this is equivalent to say that, for $C_n$, either
\begin{enumerate}\itemsep4pt
\item[$(\widetilde{A})$] $(a_0,\dots,a_n)=(n-1,0,\dots,0,1)$ or
\item[$(\widetilde B)$]  $(a_0,\dots,a_n)=(0,n,0,\dots,0)$. 
\end{enumerate}
Notice that this yields the statement of Theorem~\ref{L0Ln}. Indeed, by Remark~\ref{image of a point}, $L_0=\di(\{p\})$ for every $p\in\R^n$ and, by Claim~1, $\dim(\di_0(C_n))=\dim L_0=a_0$. Hence, we have that either $\dim(\di(\{p\}))=n-1$ or $\dim(\di(\{p\}))=0$.

\bigskip
First notice that the case $n=2$ is trivial, since by \eqref{sum=n} and \eqref{sumk=n}, the only possibilities for $(a_0,a_1,a_2)$ are $(1,0,1)$ and $(0,2,0)$. We assume in the following that $n\geq 3$. 

If $a_n=1$, by \eqref{sumk=n}, we obtain that $a_j=0$ for every $1\leq j\leq n-1$. Furthermore, using \eqref{sum=n}, $a_0=n-1$ and we are in case $(\widetilde A)$. 
If $a_n\neq 1$, then \eqref{sumk=n} yields $a_n=0$. Moreover, either $a_1=n$ and we are in case $(\widetilde B)$ or $a_1\neq n$, what the following argument proves to be impossible.

Let $C_n=S_1+\cdots+S_n$ and assume that we have $a_n=0$ and $a_1\neq n$. 
Define $$k_0:=\min\{k : a_k\neq 0, 1\le k\leq n\}.$$ Notice that, as proved at the beginning of the proof of Theorem~\ref{L0Ln}, $\dim L_0=a_0<n$. Hence by \eqref{sumk=n}, there exists $k\geq 1$ such that $a_k\neq 0$; in particular, $k_0$ is well-defined. If $k_0a_{k_0}\neq n$, define $k_1:=\min\{k : a_k\neq 0, k> k_0\}$, while 
if $k_0a_{k_0}=n$, set $k_1=k_0$. The existence of $k_1$ is guaranteed by \eqref{sumk=n}.

We claim that $k_1>1$. Indeed, if $k_1=1$, then $1=k_0=k_1$ and thus, by the definition of $k_1$, $k_0 a_{k_0}=a_1=n$, but this contradicts the assumption that $a_1\neq n$.

We claim also that $k_0<n$. Indeed,  $k_0=n$ means that $a_n\neq 0$, which by \eqref{sumk=n} is equivalent to $a_n=1$, but we are assuming $a_n=0$. 

We next claim that
\begin{equation}\label{aggiunta 2}
k_0+k_1\le n.
\end{equation}
Indeed, if $k_0<k_1$, by \eqref{sumk=n} we have $k_0+k_1\leq k_0a_{k_0}+k_1a_{k_1}\leq n$. Assume now that $k_0=k_1$; then we have  
$k_0 a_{k_0}=n$. We study the quantity $2k_0=k_0+k_1$, depending on $a_{k_0}$. If $a_{k_0}=1$, then $k_0=n$, but this is not possible, as we have
$a_n=0$. If $a_{k_0}=2$, then $2k_0=a_{k_0}k_0=n$ and \eqref{aggiunta 2} holds. Finally, if $a_{k_0}>2$ then $2k_0<a_{k_0}k_0=n$. Inequality \eqref{aggiunta 2} is proved.

Using \eqref{0or1}, \eqref{aggiunta 2}, and Claim~2(c), we may assume without loss of generality that 
\begin{equation}\label{not_zero}
\dim(\di_{k_0}(S_1+\dots+S_{k_0}))=1\quad\textrm{and}\quad\dim(\di_{k_1}(S_{k_0+1}+\dots+S_{k_0+k_1}))=1.
\end{equation}

We will apply Claim~3 to the following generalized cubes to obtain the contradiction. Consider the bases of $\R^n$ given by $\{w_1,\dots,w_{k_0},w_{k_0}+w_{k_0+1},w_{k_0+2},\dots,w_n\}$ and \\ $\{w_1,\dots,w_{k_0-1},w_{k_0}+w_{k_0+1},w_{k_0+1},\dots,w_n\}$, and the associated zonotopes 
$$\widetilde C_n:=S_1+\dots+S_{k_0}+S_{k_0,k_0+1}+S_{k_0+2}+\dots+S_n
$$
and
$$\overline C_n:=S_1+\dots+S_{k_0-1}+S_{k_0,k_0+1}+S_{k_0+1}+\dots+S_n.$$  
Here we denote $S_{k_0,k_0+1}:=[-(w_{k_0}+w_{k_0+1}),w_{k_0}+w_{k_0+1}]$. 
Observe that the choice of the first basis cannot be done if $k_0=n$, i.e., $a_n=1$, which corresponds to $(\widetilde A)$ of Claim~4.

By~\eqref{not_zero}, we have 
$$\dim(\di_{k_0}(S_1+\dots+S_{k_0}))=1.$$ 
Hence, if $k_0\geq 2$, then Claim~3(b) yields 
\begin{equation}\label{basisZero}
\dim(\di_{k_0}(S_1+\dots+S_{k_0-1}+S_{k_0,k_0+1}))=0. 
\end{equation}
If $k_0=1$ and $k_1\geq 2$, we obtain 
\begin{equation}\label{basisZero_k0=1}
\dim(\di_{k_1}(S_{1,2}+S_{3}+\dots+S_{k_1+1}))=1\quad\textrm{and}\quad\dim(\di_1(S_{1,2}))=0.
\end{equation}

Applying Claim~3(a) to the cubes $C_n$ and $P=\overline C_n$, and using \eqref{not_zero}, we have that 
$$\dim(\di_{k_0}(S_1+\dots+S_{k_0}))=1$$  
implies
\begin{equation}\label{basisNonZero}
\dim(\di_{k_0}(S_1+\dots+S_{k_0-1}+S_{k_0,k_0+1}))=1,
\end{equation}
which for $k_0=1$ means
\begin{equation}\label{basisNonZero_k0=1}
\dim(\di_{1}(S_{1,2}))=1.
\end{equation}

Hence, if $k_0\geq 2$, \eqref{basisZero} together with \eqref{basisNonZero} yields a contradiction. If $k_0=1$ and $k_1\geq 2$, then \eqref{basisZero_k0=1} with \eqref{basisNonZero_k0=1} yields also a contradiction. 
We note that there is no contradiction if $k_0=k_1=1$, which corresponds to case $(\widetilde B)$ of Claim~4, since in this case we have 
$$\dim(\di_{1}(S_1))=\dim(\di_1(S_{1,2}))=\dim(\di(S_2))=1.$$

Thus, we have proved Claim~4, which concludes the proof of the theorem.
\end{proof}

From the proof of the previous theorem, especially from Claim~1, and by approximation of arbitrary zonoids by $n$-dimensional ones, we deduce the following result.

\begin{corollary}\label{cor:dimfi}
Let $n\geq 2$ and let $\di\in\MVal$ satisfy (VC).  
\begin{enumerate}
\item[\emph{(i)}] If $\dim(\di\{0\})=\dim L_0=n-1$, then $\dim L_n=1$ and $\dim(L_0+L_n)=n$.
\item[\emph{(ii)}] If $\dim L_0=n-1$ and $Z$ is a zonoid, then $\dim(\di_j(Z))=0$ for every $j=1,\dots,n-1$.
\item[\emph{(iii)}] If $\dim L_0=0$ and $Z$ is a zonoid, then $\dim(\di_j(Z))=0$ 
for every $j=2,\dots,n$. In particular, $L_n$ is a point.
\end{enumerate}
\end{corollary}

\section{On the McMullen decomposition of valuations satisfying (VC)}\label{n-1}
In this section we will investigate more deeply the properties of the homogeneous functions in the McMullen decomposition in \eqref{h-decomp} for Minkowski valuations satisfying (VC).
The two next lemmas recall standard facts, which will be often used in the following.

\begin{lemma}\label{standard facts} Let $n\geq 2$ and $j\in\{0,1,\dots,n\}$. 
\begin{enumerate}
\item[\emph{(i)}]
If $\rv\in\Val_j$ vanishes on $j$-dimensional simplices, then $\mu$ vanishes on every $j$-di\-men\-sio\-nal convex body.
\item[\emph{(ii)}] If $\dib\in\MVal_j$ satisfies that $\dim(\dib(T))=0$ for every $j$-dimensional simplex $T$, then $\dim(\dib(K))=0$ for every $j$-dimensional convex body $K$.
\end{enumerate}
\end{lemma}
\begin{proof}
The proof of both statements follows by standard approximation arguments. Indeed, each convex body can be approximated in the Hausdorff distance by polytopes \cite[Theorem 1.8.16]{schneider.book14}. Moreover, each polytope can be decomposed in a finite number of simplices (simplicial 
decomposition) whose intersection is either empty or a lower-dimensional simplex (see e.g. \cite[Proof of Theorem 6.3.1]{schneider.book14}). The statement 
follows by using the valuation property and the continuity of the valuation. 
\end{proof}

\begin{lemma}\label{standard fact ii}
Let $n\geq 2$ and let $T$ be a $j$-dimensional simplex, $2\leq j\leq n$. Then there exists a convex polytope $P$ such that $T\cup P$ is a convex zonotope and $T\cap P$ has 
dimension $j-1$. 
\end{lemma}
\begin{proof}
Let $T$ be a $j$-dimensional simplex, $2\leq j\leq n$. Without loss of generality we may assume that one vertex of $T$ is the origin. Let $g\in \GL(n)$ be such that $g(T)$ is the standard $j$-dimensional 
simplex of the hyperplane 
$
H=\{(x_1,\dots,x_n)\,:\,x_{j+1}=\dots=x_n=0\}.
$ 

Let $C_j$ be the unit standard cube in $H$. The set $C_j\setminus g(T)$ is convex (as the intersection of $C_j$ with an open half-space of $H$) and its closure is a polytope $P$. Let $P'=g^{-1}(P)$. Then $C_j=g(T) \cup g(P')$. This shows that the desired statement holds for $g(T)$. 
Applying now $g^{-1}$, we obtain it for $T$.
\end{proof}

\begin{lemma}\label{f_i=0 esencia}
Let $n\geq 2$, $\di\in\MVal$, and $1\leq j\leq n-1$. 
Suppose that the mapping $u\mapsto f_j(Z,u)$, associated to $\di$ as defined in \eqref{h-decomp}, is a linear function for every $Z$ zonotope in $\K^n$. Then:
\begin{itemize}\itemsep9pt
\item[(i)] $u\mapsto f_j(K,u)$ is a linear function for every $K\in\K^n$ with $\dim K=j$;
\item[(ii)] if $u\mapsto f_j(K,u)$ is a support function for every $K\in\K^n$ with $\dim K=j+1$, then $u\mapsto f_j(K,u)$ is a linear function for every $K\in\K^n$ with 
$\dim K=j+1$.
\end{itemize}
\end{lemma}

\begin{proof}
Let $\di\in\MVal$, let $1\leq j\leq n-1$, and let $u\mapsto f_j(Z,u)$ be a linear function for every $Z$ zonotope in $\K^n$.
We note that, from Theorem~\ref{suppZonoid}, $f_j(Z,\cdot)$ is a support function for every zonotope $Z$ and hence we can 
write $\di_j(Z)$ for the convex body whose support function is $f_j(Z,\cdot)$. Moreover, $\dim \di_j(Z)=0$, since a convex body with linear support function is a point.
\begin{enumerate}\itemsep9pt
\item[(i)] Let $T$ be a $j$-dimensional simplex and let $P$ be a polytope given by Lemma~\ref{standard fact ii}. Then $T\cup P$ is a zonotope and, by hypothesis, $\dim(\di_j(T\cup P))=0$. 
Furthermore, since $\dim (T\cap P)=j-1$, Lemma~\ref{r: facts on f_js}(iii) yields $f_j(T\cap P,\cdot)\equiv 0$. Hence, $f_j(T\cap P,\cdot)$ is the support function of $\di_j(T\cap P)=\{0\}$. 
Moreover, from Lemma~\ref{r: facts on f_js}(iv), for every $j$-dimensional convex body $K$, $u\mapsto f_j(K,u)$ is the support function of a convex body $\di_j(K)$.  Thus, if $f_j(T,\cdot)$ and $f_j(P,\cdot)$ are the support functions of $\di_j(T)$ and $\di_j(P)$, resp., we have 
\begin{equation*}\di_j(T\cup P)=\di_j(T\cup P)+\di_j(T\cap P)=\di_j(T)+\di_j(P).\end{equation*}
Hence, $\dim(\di_j(T))=0$ for every simplex $T$ of dimension $j$. 
The statement follows by Lemma~\ref{standard facts}(ii). 

\item[(ii)] Similarly to the argument in (i), we let $T$ be a $(j+1)$-dimensional simplex and let $P$ be given by Lemma~\ref{standard fact ii}. Since $\dim(T\cap P)=j$, from the previous item we have that $\dim(\di_j(T\cap P))=0$. 
On the other hand, since $T\cup P$ is a zonotope, by hypothesis we have $\dim(\di_j(T\cup P))=0$.
Using now that $f_{j}(K,\cdot)$ is a support function for any $K\in\K^n$ with $\dim K=j+1$, we obtain
$$
\dim(\di_j(T\cup P)+\di_j(T\cap P))=\dim(\di_j(T)+\di_j(P)).
$$ 
Therefore, $\dim(\di _j(T))=0$, and Lemma~\ref{standard facts}(ii) yields the result.   
\end{enumerate}
\end{proof}

\begin{lemma}\label{513}Let $n\geq 2$ and let $\di\in\MVal$ satisfy (VC) and $\dim(\di(\{0\}))=n-1$. Then:
\begin{enumerate}
\item[\emph{(i)}] 
for every $1\leq j\leq n-1$ and $K\in\K^n$ with $\dim K=j$, $f_j(K,\cdot)$ is a linear function;
\item[\emph{(ii)}] if further $\di\in\MVal^s$, then $f_j(K,\cdot)\equiv 0$ for every $1\leq j\leq n-1$ and $K\in\K^n$ with $\dim K=j$.
\end{enumerate}
\end{lemma}

\begin{proof}
Let $\di$ be as in the statement. 
\begin{enumerate}
\item[(i)]
By Corollary~\ref{cor:dimfi}(ii), each function $u\mapsto f_j(Z,u)$, $1\leq j\leq n-1$, associated to $\di$ is the support function of a point, and hence it is a linear function. Lemma~\ref{f_i=0 esencia}(i) yields the statement. 
\item[(ii)] Let now $\di$ be also an $o$-symmetrization. By Lemma~\ref{even}, the function $u\mapsto f_j(K,u)$ is even for every $1\leq j\leq n-1$ and $K\in\K^n$. On the other hand, by item (i), we know that $u\mapsto f_j(K,u)$ is a linear function for every $K\in\K^n$ with $\dim K=j$. Both conditions imply that $f_j(K,u)=0$ for every $u\in\R^n$ and $K\in\K^n$ with $\dim K=j$.
\end{enumerate}
\end{proof}
 
\begin{lemma}\label{l: fi(K,u)=0}
Let $n\geq 2$ and let $\di\in\MVal^s$ satisfy (VC) and 
$
\dim(\di(\{0\}))=n-1.
$ 
Then $f_{j}(K,u)=0$ for every $u\in\R^n$, $1\leq j\leq n-1$, and $K\in\K^n$.
\end{lemma}

\begin{proof}
Corollary~\ref{cor:dimfi}(i) together with Lemma~\ref{even} yields $\dim L_n=1$ and $L_n\in\K^n_s$. 
Let $\{e_1,\dots,e_n\}$ be a basis of $\R^n$
such that $L_n=[-e_1,e_1]$ and denote by $H$ the $(n-1)$-dimensional subspace orthogonal to $\spa\{e_1\}$.
We divide the proof in three steps. 

\medskip

\noindent{\bf Step 1.} \emph{Let $K\in\K^n$, $1\leq j\leq n-1$, $u',v'\in H$, and $u=ae_1+u'$ and $v=be_1+v'$ with $ab\geq 0$. 
Then 
$$f_j(K,u+v)=f_j(K,u)+f_j(K,v).$$ Moreover, $f_j(K,w)=0$ for every $w\in H$.}

We prove the claim by backward induction. First we prove it for $j=n-1$.
For simplicity, we write $f(K,u)$ instead of $f_{n-1}(K,u)$. For every $a\in\R^n$,
\begin{equation}\label{first_l}
h(L_n,ae_1+u')=|\langle e_1,ae_1+u'\rangle|=|a|\|e_1\|=h(L_n,ae_1)+h(L_n,u').
\end{equation}
Since $\di(K)\in\K^n$ for every convex body $K$, for a fixed $\lambda>0$ we can write, using \eqref{h-decomp} and \eqref{first_l},
\begin{align*}
0&\geq h(\di (\lambda K),ae_1+v')-h(\di(\lambda K),ae_1)-h(\di(\lambda K),v')
\\&=\lambda^n(h(L_n,ae_1+v')-h(L_n,ae_1)-h(L_n,v'))
\\&\quad+\lambda^{n-1}(f(K,ae_1+v')-f(K,ae_1)-f(K,v'))+O(\lambda^{n-2})
\\&=\lambda^{n-1}(f(K,ae_1+v')-f(K,ae_1)-f(K,v'))+O(\lambda^{n-2}).
\end{align*}
Thus, as $\lambda\to\infty$, we obtain  
\begin{equation}\label{ineq1}
f(K,ae_1+v')\leq f(K,ae_1)+f(K,v')
\end{equation}
for every convex body $K\in\K^n$. 
Since $K\mapsto f(K,u)$ is a continuous, translation invariant, and $(n-1)$-homogeneous real-valued valuation, and, by Lemma~\ref{513}(ii), it vanishes when restricted to 
$(n-1)$-dimensional convex bodies, Lemma~\ref{th_lemma}(i) yields
\begin{equation}\label{impar}
f(K,u)+f(-K,u)=0,\quad\forall K\in\K^n,\,u\in\R^n.
\end{equation}
Combining this fact with \eqref{ineq1}, in which $K$ is replaced by $-K$, we obtain
$$f(K,a e_1+v')=f(K,a e_1)+f(K,v') \text{ for every } K\in\K^n, a\in\R, v'\in H.$$

From the fact that $h(L_n,u')=0$ we also obtain
$$h(L_n,u+v)=h(L_n,(a+b)e_1)=|a+b|\|e_1\|$$
and 
$$h(L_n,u)+h(L_n,v)=h(L_n,ae_1)+h(L_n,be_1)=(|a|+|b|)\|e_1\|,$$
which yields, for $a,b\in\R$ with the same sign, 
\begin{equation}\label{second_l}
h(L_n,u+v)=h(L_n,u)+h(L_n,v).
\end{equation}
Using, as above, that $\di(\lambda K)\in\K^n$ for every $K\in\K^n$ and $\lambda>0$ together with \eqref{second_l} and \eqref{impar},  
we obtain 
$$f(K,u+v)=f(K,u)+f(K,v).$$
If we apply this equality to $u=e_1+w$ and $v=e_1-w$, $w\in H$, using \eqref{first_l} and that, by Lemma \ref{even}, $f$ is even, 
we have
\begin{align*}
2f(K,e_1)&=f(K,u+v)
\\&=f(K,u)+f(K,v)
\\&=2f(K,e_1)+f(K,w)+f(K,-w)
\\&=2(f(K,e_1)+f(K,w)),
\end{align*}
which implies $f(K,w)=0$ for every convex body $K\in\K^n$ and $w\in H$. 
Hence, we have proved Step~1 for $j=n-1$.

In order to proceed with the (backward) induction, we assume that the claim holds for $j> j_0$ and  prove it for $j=j_0$. By the induction hypothesis, and the McMullen decomposition in \eqref{h-decomp}, we can argue for $f_{j_0}$ as we have just done with $f$ to prove the statement. 

\medskip
\noindent{\bf Step 2.} \emph{For every $K\in\K^n$ and $1\leq j\leq n-1$, either the function $u\mapsto f_j(K,u)$ or $u\mapsto f_j(-K,u)$ is a support function and 
\begin{equation}\label{aggiunta1}
f_j(K,u)=(-1)^{\varepsilon_j(K)}\alpha_j(K)h([-e_1,e_1],u)=(-1)^{\varepsilon_j(K)}\alpha_j(K)|\langle e_1,u\rangle|
\end{equation}
where $\varepsilon_j(K)\in\{0,1\}$ and $\alpha_j(K)\geq 0$.}

Let $u=ae_1+u',v=be_1+v'\in\R^n$ with $a,b\in\R$ and $u',v'\in H$. Step~1 and the evenness of $u\mapsto f_j(K,u)$ yield
\[
\begin{split}
\quad f_j(K,u+v)&=f_j(K,(a+b)e_1)+f_j(K,u'+v')\\&=|a+b|f_j(K,\sign(a+b)e_1)=|a+b|f_j(K,e_1) \quad\text{and}
\\ f_j(K,u)+&f_j(K,v)=(|a|+|b|)f_j(K,e_1).
\end{split}
\]
Let $K\in\K^n$ be such that $f_j(K,e_1)\geq 0$. As a consequence of the previous equalities,
$$
f_j(K,u+v)\leq f_j(K,u)+f_j(K,v),\quad\forall\, u,v\in\R^n.
$$
This means that $u\mapsto f_j(K,u)$ is a support function. If $f_j(K,e_1)<0$, we can use 
Lemmas~\ref{513}(ii) and \ref{th_lemma}(i) to obtain that 
$f_j(-K,e_1)> 0$. Now, applying the 
previous argument to $-K$, we get that $u\mapsto f_j(-K,u)$ is a support function.

Let $1\leq j\leq n-1$ and let $K\in\K^n$ be such that $u\mapsto f_j(K,u)$ is a support function. Let $\di_j(K)\in\K^n$ be such that $f_j(K,\cdot)=h(\di_j(K),\cdot)$. From Step~1, $\di_j(K)$ lies on the line orthogonal to $H$ passing through the origin. Since $\di_j(K)\in\K^n_s$ is a centered convex body, it is a centered segment on the line 
spanned by $e_1$. Thus, there exists $\alpha_j(K)\geq 0$ (depending on $K$ and $j$) such that $\di_j(K)=\alpha_j(K)[-e_1,e_1]$.
Using $f_j(K,u)=-f_j(-K,u)$, we get \eqref{aggiunta1}. 

\medskip
\noindent{\bf Step 3.} \emph{For every $K\in\K^n$ and $j\in\{1,\dots, n-1\}$, $f_j(K,\cdot)\equiv 0$.}

We prove it by induction on $j$. 
Let $j=1$. Let $K\in\K^n$ be so that $\dim K\geq 1$ and $f_1(K,\cdot)$ is a support function.  By Corollary~\ref{mixedZero}(i), the mixed volume $V(L_0[n-1],f_1(K,\cdot))$ vanishes since $V(L_0[n-1],f_1(K,\cdot))=v_1^{\di}(K)$, that is, $V(L_0[n-1],f_1(K,\cdot))$ is the coefficient of the 1-homogeneous term of the polynomial in \eqref{vi di}. On the other hand,
$$
V(L_0[n-1],f_1(K,\cdot))=(-1)^{\varepsilon(K)}\alpha_1(K)V(L_0[n-1],[-e_1,e_1]),
$$
which, by Corollary~\ref{cor:dimfi}(i), vanishes if and only if $\alpha_1(K)=0$. If $\dim K=0$, then $f_1(K,u)=0$, for every $u\in\R^n$, by Lemma~\ref{r: facts on f_js}(iii). Hence, we have $f_1(K,\cdot)\equiv 0$ for every $K\in\K^n$.

Assume that $f_j(K,\cdot)$ vanishes for every $j<j_0\le n-1$ and for every $K\in\K^n$. Let $K\in\K^n$ have $\dim K\geq j_0$ and consider 
$v_{j_0}^{\di}(K)$, i.e., the coefficient of degree $j_0$ in the polynomial expansion \eqref{vi di}. 
Remark~\ref{f_i in v_j} yields that the only possible entries of each mixed volume summand of $v_{j_0}^{\di}(K)$ are $f_j(K,\cdot)$ with $0\leq j\leq j_0$. Therefore, 
by the induction hypothesis, $v_{j_0}^{\di}(K)$ is given only by the summand $V(L_0[n-1],f_{j_0}(K,\cdot))$. As $j_0<n$, by Corollary~\ref{mixedZero}(i), 
$$
v_{j_0}^{\di}(K)=V(L_0[n-1],f_{j_0}(K,\cdot))=0.
$$
The latter is true if and only if $\alpha_{j_0}(K)=0$, by a similar argument as for $j=1$. Hence, the statement of Step~3 and so, also Lemma~\ref{l: fi(K,u)=0} are proved. 
\end{proof}

\section{Proof of Theorems \ref{teo} and \ref{cor}}\label{1}
We start with the following theorem in which we give an explicit expression for the image of an operator $\di\in\MVal$ satisfying (VC) and such that $\dim(\di(\{0\}))=0$. 

\begin{theorem}\label{di_1Z=n}Let $n\geq 2$. An operator $\di\in\MVal$ satisfies (VC) and $\dim(\di(\{0\}))=0$ if and only if 
$$
\di(K)=p+\di_1(K)+p_2(K)+\dots+p_{n-1}(K)+\vol(K)q,\quad\forall K\in\K^n,
$$
with $p,q\in\R^n$, $p_j:\K^n\longrightarrow\R^n$, $2\leq j\leq n-1$, continuous, translation invariant, and $j$-homogeneous valuations, and 
$\di_1\in\MVal_1$ satisfying (VC). 
\end{theorem}  

\begin{proof}
Assume first that $\di\in\MVal$ satisfies the hypotheses of the statement. 
We show that $f_j(K,\cdot)$ is a linear function for every $K\in\K^n$ and $2\leq j\leq n$, by backward induction on $j$. 

By Corollary~\ref{cor:dimfi}(iii), if $K$ is a zonoid, then $f_j(K,\cdot)$ is a linear function for every $2\leq j\leq n$. Thus, Lemma~\ref{f_i=0 esencia}(i), ensures that if $K$ is a convex body with $\dim K=j$, $2\leq j\leq n-1$, then $f_j(K,\cdot)$ is linear. Furthermore, since 
$h(L_n,\cdot)$ is a linear function, Lemma~\ref{r: facts on f_js}(vi) and Lemma~\ref{f_i=0 esencia}(ii) yield $\dim(\di_{n-1}(K))=0$ for every $K\in\K^n$ of dimension $n$, that is, $f_{n-1}(K,\cdot)$ in the McMullen decomposition of $\di(K)$ is also a linear function.
We now proceed with a backward induction argument. Let us assume that $f_{j}(K,\cdot)$ is a linear function for every $j_0< j\leq n-1$ and $K\in\K^n$. By Lemma~\ref{r: facts on f_js}(vi) and Lemma~\ref{f_i=0 esencia}(ii),  
$f_{j_0}(K,\cdot)$ is linear for every $K\in \K^n$ with $\dim K=j+1$. Theorem~\ref{th_lemma} yields that $f_j(K,\cdot)$ is linear for every $K\in\K^n$, $2\leq j\leq n-1$.  (Notice that in general $f_1(K,\cdot)$ is not linear since it is not linear for zonotopes.)

Now, again by Lemma~\ref{r: facts on f_js}(vi), we have that $u\mapsto f_1(K,u)$ is a support function for every $K\in\K^n$. We denote by $\di_1(K)$ the convex body such that $h(\di_1(K),\cdot)=f_1(K,\cdot)$. The map $K\mapsto\di_1(K)$ is a continuous and translation invariant Minkowski valuation that satisfies $\vol(\di(K))=\vol(\di_1(K))$. Therefore, $\di_1\in\MVal_1$ satisfies~(VC). 

The converse is clear. 
\end{proof}

Next we give the explicit expression for the image of an operator $\di\in\MVal$ satisfying (VC) and such that $\dim(\di(\{0\}))=n-1$. 

\begin{theorem}\label{di_1Z=0_noOSym}Let $n\geq 2$. An operator $\di\in\MVal$ satisfies (VC) and $\dim(\di(\{0\}))=n-1$ 
if and only if there exist $L\in \K^n$ with $\dim L=n-1$ and 
a segment $S$ with $\dim (L+S)=n$ such that
\begin{equation}\label{di_K}
\di(K)=L+p_1(K)+\dots+p_{n-1}(K)+\vol(K)S,\quad\forall K\in\K^n,
\end{equation}
where $p_j:\K^n\longrightarrow\R^n$ is a continuous, translation invariant valuation, homogeneous of degree $j$, $1\leq j\leq n-1$. 
\end{theorem}  

In order to prove Theorem \ref{di_1Z=0_noOSym}, we will first prove its symmetric version, namely, the following result.

\begin{theorem}\label{di_1Z=0}Let $n\geq 2$. An operator $\di\in\MVal^s$ satisfies (VC) and $\dim(\di(\{0\}))=n-1$ if and only if there exist $L\in \K^n_s$ with $\dim L=n-1$ and a 
centered segment $S$ with $\dim(L+S)=n$ such that
$$
\di(K)=L+\vol(K)S,\quad\forall K\in\K^n.
$$
\end{theorem}  

\begin{proof}
By Lemma~\ref{l: fi(K,u)=0} and the McMullen decomposition \eqref{h-decomp}, there exist
$L_0, L_n\in\K^n$ such that
$$
\di(K)=L_0+V_n(K)L_n,\quad\forall\, K\in\K^n.
$$
On the other hand, by Corollary~\ref{cor:dimfi}(i) we have that $\dim L_0=n-1$, 
$\dim L_n=1$, and $\dim(L_0+L_n)=n$.

The converse clearly holds since $K\mapsto L+\vol(K)S$, with $L,S$ as assumed, satisfies all the conditions (cf. Section~\ref{sec:volume_constraints}).
\end{proof}

Now we proceed with the proof of Theorem~\ref{di_1Z=0_noOSym}.

\begin{proof}[Proof of Theorem~\ref{di_1Z=0_noOSym}] Let $\di\in\MVal$ satisfy (VC) and $\dim(\di(\{0\}))=n-1$. Define the operator $\dib:\K^n\longrightarrow\K^n_s$ by
$$
\dib(K):=D(\di(K)).
$$
It is clear that $\dib$ is a continuous and translation invariant Minkowski valuation which satisfies (VC) as a consequence of (RS) and the assumption that $\di$ satisfies (VC). Moreover, the image of $K$ under $\dib$ is an $o$-symmetric convex body, since the 
difference body operator has this property. Notice also that $\dim(\dib(\{0\}))=\dim(D(\di(\{0\})))=\dim(\di(\{0\}))=n-1$, since the difference body operator preserves the dimension of any convex body.
Hence, we can apply Theorem~\ref{di_1Z=0} to $\dib$ and obtain the existence of an $(n-1)$-dimensional $o$-symmetric convex
body $L$ and a centered segment $S$ such that
\begin{equation}\label{dii_L_S}
\dib(K)=L+\vol(K)S,\quad\forall K\in\K^n.
\end{equation}
If we write it in terms of the support function, and for a $\lambda>0$, we have
\begin{equation}\label{sup_psi}
h(\dib(\lambda K),u)=h(L,u)+\lambda^n\vol(K)h(S,u).
\end{equation}

The support function of $\dib(K)$ can be also written in terms of the support function of $\di(K)$. By the homogeneity of each summand in the McMullen 
decomposition of $\di$ in \eqref{h-decomp}, we have 
\begin{align}\label{sup_psi2}\nonumber
h(\dib(\lambda K),u)&=h(\di(\lambda K),u)+h(\di(\lambda K),-u)\\&=\,h(L_0,u)+h(L_0,-u)+\lambda(f_1(K,u)+f_1(K,-u))+\dots+\nonumber\\
&\,\,\,\,\,\,+\lambda^{n-1}(f_{n-1}(K,u)+f_{n-1}(K,-u))+\lambda^n\vol(K)(h(L_n,u)+h(L_n,-u)),
\end{align}
for every $u\in\R^n$. Comparing the coefficients of the polynomials in \eqref{sup_psi} and \eqref{sup_psi2}, we obtain that $L_n$ is a segment in the same 
direction as $S$ and that $L_0$ is an $(n-1)$-dimensional convex body lying in a parallel hyperplane to $\spa L$. Moreover,  
\begin{equation}\label{f_i=0}
f_j(K,u)+f_j(K,-u)=0,\quad\forall\, 1\leq j\leq n-1,\,u\in\R^n,\,K\in\K^n.
\end{equation}

Our aim is to show 
\begin{equation}\label{linear_K}
f_j(K,u+v)=f_j(K,u)+f_j(K,v),\quad \forall\, 1\leq j\leq n-1, u,v\in\R^n, K\in\K^n.
\end{equation}
Once it is proved, we have that $\di(K)$ is given as in~\eqref{di_K}, since the functions $u\mapsto f_j(K,u)$ are linear functions, i.e., $f_j(K,u)=h(\{p_j(K)\},u)$, as we want to show.

To prove \eqref{linear_K}, we use the following two claims.

\medskip

\noindent{\bf Claim 1.} \emph{Let $1\leq j\leq n-1$. Let $T$ be a $(j+1)$-dimensional simplex and let $P$ be a polytope given by Lemma~\ref{standard fact ii}. Then 
$u\mapsto f_{j}(T,u)+f_{j}(P,u)$ is a linear function in $\R^n$.}

By Lemma~\ref{standard fact ii}, $T\cup P$ is a zonotope. Thus, by Corollary~\ref{cor:dimfi}(ii),  $\dim(\di_j(T\cup P))=0$, that is, 
$u\mapsto f_j(T\cup P,u)=h(\di_j(T\cup P),u)$ is a linear function. By Lemma~\ref{r: facts on f_js}(iv), $f_j(T\cap P,\cdot)$ is a support function. Now Lemma~\ref{513} yields $\dim(\di_j(T\cap P))=0$. Hence, as $K\mapsto f_j(K,u)$ is a valuation for every $u\in\R^n$,
$$
f_j(T,u)+f_j(P,u)=h(\di_j(T\cup P),u)+h(\di_j(T\cap P),u)=\langle q_{_{T,P}},u\rangle, 
$$
for some $q_{_{T,P}}\in\R^n$. In other words, $u\mapsto f_j(T,u)+f_j(P,u)$ is a linear function.

\medskip

\noindent{\bf Claim 2.} 
\emph{Let $\{e_1,\dots,e_n\}$ be a basis of $\R^n$ such that $S=[-e_1,e_1]$ and denote by $H$ the hyperplane orthogonal to $S$. 
Then, for every $K\in\K^n$, $u',v'\in H$, and $1\leq j\leq n-1$,
\begin{equation}\label{eq_K}
f_j(K,u'+v')=f_j(K,u')+f_j(K,v')
\end{equation}
and 
\begin{equation}\label{eq_K_a}
f_j(K,ae_1+u')=f_j(K,ae_1)+f_j(K,u'), \quad\forall\,a\in\R.
\end{equation}}

We prove \eqref{eq_K} by backward induction on $j$.  Assume first $j=n-1$. 
We argue as in Step~1 of Lemma~\ref{l: fi(K,u)=0}. Since $\di(K)\in\K^n$ for every convex body $K$, for $\lambda>0$, we have 
\begin{align*}
0&\geq h(\di (\lambda K),u'+v')-h(\di(\lambda K),u')-h(\di(\lambda K),v')
\\&=\lambda^{n-1}(f_{n-1}(K,u'+v')-f_{n-1}(K,u')-f_{n-1}(K,v'))+O(\lambda^{n-2}).
\end{align*}
As $\lambda\to\infty$, we obtain 
\begin{equation}\label{ineq_n-1}
f_{n-1}(K,u'+v')\leq f_{n-1}(K,u')+f_{n-1}(K,v'),\quad \forall\, K\in\K^n,u',v'\in H.
\end{equation}
In order to obtain equality in \eqref{ineq_n-1}, we first apply \eqref{ineq_n-1} to a simplex $T$ and a polytope $P$ satisfying the condition of the previous claim with $j=n-1$ and add both expressions, to obtain 
\begin{equation}\label{ineq_eq}f_{n-1}(T,u'+v')+f_{n-1}(P,u'+v')\leq f_{n-1}(T,u')+f_{n-1}(T,v')+ f_{n-1}(P,u')+f_{n-1}(P,v').
\end{equation}
Now Claim~1 yields that both sides of the above inequality are the same linear function. Hence, we have equality in~\eqref{ineq_eq}, which together with inequality~\eqref{ineq_n-1} 
yields 
$$f_{n-1}(T,u'+v')= f_{n-1}(T,u')+f_{n-1}(T,v')$$ 
for every $(n-1)$-dimensional simplex $T$. As $K\mapsto f_{n-1}(K,u'+v')-f_{n-1}(K,u')-f_{n-1}(K,v')$ is a continuous and  translation invariant real-valued valuation for every $u'\in H$, Lemma~\ref{standard facts}(i) yields \eqref{eq_K} for $j=n-1$.

Assuming next that \eqref{eq_K} holds for every $j> j_0$, we show it for $j=j_0$. In this case, we obtain, similarly to the previous case, 
$f_{j_0}(K,u'+v')\leq f_{j_0}(K,u')+f_{j_0}(K,v')$ for every $K\in\K^n$ and $u',v'\in H$. Hence, applying again Claim~1, now for $j=j_0$, and Lemma~\ref{standard facts}(i), we get \eqref{eq_K} for every $1\leq j\leq n-1$. 

The proof of \eqref{eq_K_a} follows, similarly, by a backward induction argument on $j$. Indeed, since we have proven that $L_n$ is a segment in the same direction as $S$ (see~\eqref{dii_L_S}), we have $h(L_n,ae_1+v)=h(L_n,ae_1)+h(L_n,v)$. Since $\di(K)$ is a convex body, 
$$f_{n-1}(K,ae_1+u')\leq f_{n-1}(K,ae_1)+f_{n-1}(K,u')$$ for every $K\in\K^n$, $u'\in H$ and $a\in\R$ (arguing as we did for \eqref{ineq_n-1}). Now, exactly in the same way we performed the proof of \eqref{eq_K}, Claim~1 and Lemma~\ref{standard facts}(i) ensure \eqref{eq_K_a}. Thus, Claim~2 is proved.

\medskip
Now we proceed to prove \eqref{linear_K}. 
Let $u=ae_1+u'$ and $v=be_1+v'$, $a,b\in\R$, $u',v'\in H$. We compute, by using \eqref{eq_K_a},
$$
f_j(K,u+v)=f_j(K,(a+b)e_1+u'+v')=|a+b|f_j(K,\mathrm{sgn}(a+b)e_1)+f_j(K,u'+v')
$$
and 
$$
f_j(K,u)+f_j(K,v)=|a|f_j(K,\mathrm{sgn}(a)e_1)+f_j(K,u')+|b|f_j(K,\mathrm{sgn}(b)e_1)+f_j(K,v').
$$
Assume that $a+b>0$, $a>0$, and $b<0$. By using the above equations,  \eqref{eq_K}, and \eqref{f_i=0}, we get
\begin{align*}
f_j(K,u+v)&=(a-|b|)f_j(K,e_1)+f_j(K,u'+v')
\\&=af_j(K,e_1)-|b|f_j(K,e_1)+f_j(K,u')+f_j(K,v')
\\&=af_j(K,e_1)+f_j(K,u')+|b|f_j(K,-e_1)+f_j(K,v')
\\&=f_j(K,u)+f_j(K,v),
\end{align*}
for every $1\leq j\leq n-1$ and $K\in\K^n$. The equality for the remaining cases (different signs of $a,b$) is obtained in a similar way.
Hence, \eqref{linear_K} is proved. 

The converse is clear, as for any $K\in\K^n$, $K\mapsto L+p_1(K)+\dots+p_{n-1}(K)+\vol(K)S$, satisfies all the stated conditions (cf. Section~\ref{sec:volume_constraints}).
\end{proof}

\begin{proof}[Proof of Theorems \ref{teo} and \ref{cor}]
First we note that Theorem~\ref{cor} follows from Theorem~\ref{teo} and the assumption that $\di$ is an $o$-symmetrization, since the only point which is $o$-symmetric is the origin.
The proof of Theorem~\ref{teo} follows from Theorem~\ref{L0Ln}, together with Theorem~\ref{di_1Z=n}, for the case (i), and Theorem~\ref{di_1Z=0_noOSym}, for the case (ii).
\end{proof}

We note that the operators in Theorem~\ref{cor}(ii) are $\SL(n)$-invariant. It is well-known that the continuous and translation invariant real-valued valuations, which are $\SL(n)$-invariant, are linear combinations of the Euler characteristic and the volume. From this fact, it easily follows that the continuous and translation invariant Minkowski valuations, which are $\SL(n)$-invariant, are of the form $K\mapsto M_1+\vol(K)M_2$, where $M_1,M_2\in\K^n$ are fixed. The above result characterizes the $\SL(n)$-invariant Minkowski valuations satisfying (VC).

\begin{corollary}
Let $n\geq 2$. An operator $\di:\K^n\longrightarrow \K^n_s$ is a continuous, $\SL(n)$-invariant, and translation invariant Minkowski valuation satisfying (VC) if and only if 
there exist a centered segment $S$ and an $o$-symmetric $(n-1)$-dimensional convex body $L$ with $\dim (L+S)=n $ such that for every $K\in\K^n$, 
$$\di(K)= L+\vol(K)S.$$
\end{corollary}

\section{Proof of Theorems~\ref{+mon intro2} and \ref{+On_dim_geq_32}}\label{on_mon}
In this section, we apply Theorem~\ref{teo} to obtain Theorems~\ref{+mon intro2} and \ref{+On_dim_geq_32}. These are improvements of the following results from~\cite{abardia.colesanti.saorin1}, as the homogeneity hypothesis is removed.

\begin{theorem}[\cite{abardia.colesanti.saorin1}]\label{mon}
Let $n\geq 2$. An operator $\di\in\MVal$ is $1$-homogeneous, monotonic, and satisfies (VC) if and only if there is a $g\in\GL(n)$ such that 
$$
\di(K)=g(DK),\quad\forall\, K\in\K^n.
$$
\end{theorem}

\begin{theorem}[\cite{abardia.colesanti.saorin1}]\label{+On_dim_geq_3}
Let $n\geq 3$. 
\begin{enumerate}
\item[\emph{(i)}] An operator $\di\in\MVal$ is $1$-homogeneous, $\SO(n)$-covariant, and satisfies (VC) if and only if there are $a,b\geq 0$ with $a+b>0$ such that 
$$
\di(K)=a(K-\st(K))+b(-K+\st(K)),\quad\forall K\in\K^n.
$$

\item[\emph{(ii)}] An operator $\di\in\MVal^{s}$ is $1$-homogenous, $\SO(n)$-covariant, and satisfies (VC) if and only if there is a $\lambda >0$ such that $\di(K)=\lambda DK$ for every $K\in\K^n$.
\end{enumerate}
\end{theorem} 

We consider first Theorem~\ref{mon} and prove that the homogeneity property can be removed, that is, we prove Theorem~\ref{+mon intro2}. 

\begin{proof}[Proof of Theorem~\ref{+mon intro2}]
By Theorem~\ref{teo}, we have that $\di$ is either of the form
\begin{equation}\label{+On_1}
\di(K)=p+\di_1(K)+p_2(K)+\dots+p_{n-1}(K)+\vol(K)q,\quad\forall\,K\in\K^n,
\end{equation}
with $\di_1\in\MVal_1$, $p,q\in\R^n$ and $p_j:\K^n\longrightarrow\R^n$, $2\leq j\leq n-1$, continuous, translation invariant, and $j$-homogeneous valuations; 
or
\begin{equation}\label{+On_2}
\di(K)= L+p_1(K)+\dots+p_{n-1}(K)+\vol(K)S,\quad\forall\,K\in\K^n,
\end{equation}
with $S$ a non-degenerated segment, $L$ an $(n-1)$-dimensional convex body such that $\dim (L+S)=n $ and $p_j:\K^n\longrightarrow\R^n$, $1\leq j\leq n-1$, continuous, 
translation invariant, and $j$-homogeneous valuations. 

We observe first, that the monotonicity condition implies that for every $K\in\K^n$ and for every $\lambda\geq 1$, 
$\di(K)\subset\di(\lambda K)$, and that for every $0<\lambda \leq 1$, we have $\di(\lambda K)\subset\di(K)$ (notice that, by translation invariance we may assume
that $K$ contains the origin, so that $K\subset \lambda K$ for $\lambda\ge1$ and $K\supset\lambda K$ for $0<\lambda\le1$).

First we deal with the case of $\di$ being given as in \eqref{+On_1}. Let $\lambda\geq 1$ and $K\in\K^n$. Applying the support function to both sides of \eqref{+On_1}, since $h(\{p\},u)=\langle p,u\rangle$ for any $p\in\R^n$, the monotonicity condition $\di(K)\subset\di(\lambda K)$ implies that 
$$h(\di_1(K),u)+\sum_{j=2}^n\langle p_j(K),u\rangle\leq \lambda h(\di_1(K),u)+\sum_{j=2}^n\lambda^j\langle p_j(K),u\rangle,
$$
for any $u\in\R^n$, $\lambda\geq 1$, and $K\in\K^n$.
Following the notation in \eqref{+On_1}, we set $p_n(K)=q$.  
As $\lambda\to\infty$, we obtain that $\langle p_n(K),u\rangle\geq 0$ for every $u\in\R^n$, which is possible only if $q=p_n(K)=0$ for every $K\in\K^n$. 
By backward induction, as $\lambda\to\infty$, we obtain that $p_j(K)=0$ for every $K\in\K^n$ and $2\leq j\leq n-1$. 
Hence, $\di$ as given in \eqref{+On_1} is monotonic if and 
only if $\di=\di_1+p$, where $\di_1\in\MVal_1$ is monotonic. Theorem~\ref{mon} yields the first statement of Theorem~\ref{+mon intro2}.

We now deal with the operators of the form \eqref{+On_2}. Again taking the support function, we get, for every $u\in\R^n$, $0\leq\lambda\leq 1$, and $K\in\K^n$, 
$$\sum_{j=1}^{n-1}\langle p_j(K),u\rangle+\vol(K)h(S,u)\geq \sum_{j=1}^{n-1}\lambda^j\langle p_j(K),u\rangle+\lambda^n\vol(K)h(S,u).
$$
As $\lambda\to 0^+$, we have $\langle p_1(K),u\rangle\geq 0$ for every $u\in\R^n$, which implies $p_1(K)=0$ for every $K\in\K^n$. Induction on $j$ yields $p_j(K)=0$ for every 
$1\leq j\leq n-1$. Hence, the above inequality holds if $\lambda^nh(S,u)\leq h(S,u)$ for every $u\in\R^n$ and $0\leq \lambda\leq 1$. This is possible only if $h(S,u)\geq 0$ for every 
$u\in\R^n$, that is, if $S$ contains the origin. The result follows after observing that $K\mapsto L+\vol(K)S$ is monotonic if $S$ contains the origin.
\end{proof}

Next we prove Theorem~\ref{+On_dim_geq_32}. For that we apply Theorem~\ref{+On_dim_geq_3} and the following characterization, by Schneider \cite{schneider72}, of the Steiner point. We recall that the \emph{Steiner point} $\st(K)$ of $K\in\K^n$ is defined as
$$\st(K)=\frac{1}{\kappa_n}\int_{\sfe}h(K,u)udu$$ and that an operator $\vv:\K\longrightarrow \R^n$ is translation covariant if
$\vv(K+t)=\vv(K)+t$ for every $t\in\R^n$.
We observe that the Steiner point is a continuous, translation covariant, $\SO(n)$-covariant, and 1-homogeneous vector-valued valuation (see \cite[p.~50]{schneider.book14}). 
Schneider first proved in~\cite{schneider71} that this list of conditions characterizes the Steiner point. In \cite{schneider72}, he removed the homogeneity hypothesis and proved the following result. 

\begin{theorem}[\cite{schneider72}]\label{schneider_vector} An operator $\vv:\K^n\longrightarrow\R^n$ is a continuous, $\SO(n)$-covariant, and translation covariant valuation if and only if there is a $\lambda\geq 0$ such that $\vv(K)=\lambda\,\st(K)$ for every $K\in\K^n$.
\end{theorem}

\begin{proof}[Proof of Theorem \ref{+On_dim_geq_32}]
We will show that every operator $\di\in\MVal$ satisfying (VC) and being $\SO(n)$-covariant is also 1-homogeneous, so that we can apply Theorem~\ref{+On_dim_geq_3}.

By Theorem \ref{teo}, we know that $\di$ is either of the form \eqref{+On_1} or \eqref{+On_2}.
We study which of those operators are $\SO(n)$-covariant. 

Assume first that $\di$ is given as in \eqref{+On_2}. Applying \eqref{+On_2} to $\lambda K$, for $K\in\K^n$ and $\lambda>0$, taking support functions in \eqref{+On_2}, and using the $\SO(n)$-covariance, we have for every 
$g\in\SO(n)$ and $u\in\R^n$, 
\begin{eqnarray*}
&&h(L,u)+\sum_{j=1}^{n-1}\lambda^j h(p_j(g(K)),u)+\lambda^n\vol(K)h(S,u)\\
&&\quad\qquad=h(g(L),u)+\sum_{j=1}^{n-1}\lambda^j h(g(p_j(K)),u)+\lambda^n\vol(K)h(g(S),u).
\end{eqnarray*}
As $\lambda\to 0^+$, we obtain $h(L,u)=h(g(L),u)$ for every $g\in\SO(n)$, where $L\in\K^n$ is a fixed $(n-1)$-dimensional convex body. Since $h(L,u)=h(g(L),u)$ holds for every $g\in\SO(n)$  only if $L=\{0\}$ or $L=rB^n$, $r>0$, and non of these are $(n-1)$-dimensional convex bodies, we obtain that the case given by \eqref{+On_2} does not contain any $\SO(n)$-covariant valuation. 

Similarly, from \eqref{+On_1}, we have, for every $u\in\R^n$, $g\in\SO(n)$, $\lambda>0$, and $K\in\K^n$, 
\begin{eqnarray*}
&&h(\{p\},u)+\lambda h(\di_1(g(K)),u)+\sum_{j=2}^{n}\lambda^j h(\{p_j(g(K))\},u)\\
&&\quad\qquad=h(\{g(p)\},u)+\lambda h(g(\di_1(K)),u)+\sum_{j=2}^{n}\lambda^j h(\{g(p_j(K))\},u),
\end{eqnarray*}
which implies 
$$
h(\{p_j(g(K))\},u)=h(\{g(p_j(K))\},u),\quad\forall u\in\R^n,\, g\in\SO(n),\, K\in\K^n,\, 2\leq j\leq n.
$$
Hence, since the support functions of $p_j(K)\in\R^n$ and $g(p_j(K))\in\R^n$ coincide, we have
$p_j(g(K))=gp_j(K)$, $2\leq j\leq n$. By Theorem \ref{schneider_vector} applied to $K\mapsto p_j(K)-\st(K)$, which is translation covariant, this is possible only if $p_j(K)=\lambda_j\,\st(K)$ for every $K\in\K^n$ and some $\lambda_j\geq 0$. As $K\mapsto p_j(K)$ is homogeneous of degree $j$, $2\leq j\leq n$, and $K\mapsto\st(K)$ is homogeneous of degree 1, this is the case only if 
$\lambda_j=0$, for every $2\leq j\leq n$. 
Similarly, $h(\{p\},u)=h(\{g(p)\},u)$ 
for every $g\in\SO(n)$ and $u\in\R^n$, implies $p=0$. Therefore, the only operators of the form in \eqref{+On_1} which are $\SO(n)$-covariant are 1-homogeneous. 

We can now apply Theorem~\ref{+On_dim_geq_3} to obtain the result. 
\end{proof}

In a similar manner, we can use Theorem~\ref{teo} and Theorem~6.1 in \cite{abardia.colesanti.saorin1} to show the 2-dimensional version of Theorem~\ref{+On_dim_geq_32}.
\begin{theorem}
Let $n=2$. An operator $\di\in\MVal$ satisfies (VC) and is $\SO(2)$-covariant if and only if there are $g\in\SO(2)$ and $a,b\geq 0$ with $a+b>0$ such that 
$$
\di(K)=ag(K-\st(K))+bg(-K+\st(K)),\quad\forall K\in\K^n.
$$
\end{theorem}

\section{Examples}\label{examples}

We provide some examples of operators satisfying all but one of the hypothesis of Theorem~\ref{teo} and hence showing that the result is best possible, in the sense that other operators appear if one of the hypothesis is removed, except for the continuity. We are not aware of a translation invariant Minkowski valuation satisfying (VC) which is not continuous.

\begin{example}
Let $L$ be an $(n-1)$-dimensional convex body and let $S$ be a segment such that $\dim (L+S)=n$. Then
\[
K\mapsto DK+\st(K)
\]
and 
\[
K\mapsto L+\vol(K) S +\st(K)
\]
are continuous Minkowski valuations which satisfy (VC).  However, they are not translation invariant,
since the Steiner point is not. 
\end{example}

\begin{example}
The operator \[
K\mapsto {\rm conv}\left((K-\st(K))\cup (-K+\st(K))\right)
\]
is continuous, translation invariant, and satisfies (VC).
It is also an $o$-symmetrization. However, it is not a Minkowski valuation.
\end{example}

\begin{example}
Let $L$ be an $(n-1)$-dimensional convex body and let $S$ be a segment such that $\dim (L+S)=n$.
Then 
\[
K\mapsto L+\vol(DK) S
\]
is a continuous and translation invariant operator satisfying (VC). However, it is not a Min\-kows\-ki valuation.
\end{example}

\begin{example}
For $n\geq 2$, the complex difference body introduced in \cite{abardia}, $D_C:\K^{2n}\longrightarrow \K^{2n}$,
with $C$ an $o$-symmetric planar convex body
provides a continuous and translation invariant Minkowski valuation that satisfies (LVC) and is an $o$-symmetrization. 
However, it does not satisfy (UVC).
\end{example}

\begin{example}
Let $L$ be an $(n-1)$-dimensional convex body and let $S$ be a segment with $\dim (L+S)=n$.
Then the operator
\[
K\mapsto L+\vol(K) S+ DK
\]
is a continuous and translation invariant Minkowski valuation satisfying (LVC). However, it does not satisfy (UVC).
\end{example}

\begin{example}
Let $L$ be an $(n-1)$-dimensional symmetric convex body and let $S$ be a segment with $\dim (L+S)<n$.
Then the operator
\[
K\mapsto L+\vol(K) S
\]
is a continuous and translation invariant Minkowski valuation satisfying (UVC). However, it does not satisfy (LVC).
\end{example}


\begin{thebibliography}{99}
\bibitem{abardia.bernig}
Judit Abardia and Andreas Bernig.
\newblock Projection bodies in complex vector spaces.
\newblock {\em Adv. Math.}, 227(2):830--846, 2011.

\bibitem{abardia}
Judit Abardia.
\newblock Difference bodies in complex vector spaces.
\newblock \emph{J. Funct. Anal.}, 263(11):3588--3603, 2012.

\bibitem{abardia.colesanti.saorin1}
Judit Abardia-Ev\'equoz, Andrea Colesanti, and Eugenia Saor\'in G\'omez.
\newblock Minkowski additive operators under volume constraints.
\newblock \emph{Available at} arXiv:1702.04237.

\bibitem{abardia.saorin2}
Judit Abardia-Ev\'equoz and Eugenia Saor\'in G\'omez. 
\newblock The role of the Rogers-Shephard inequality in the characte\-rization of the difference body. 
\newblock \emph{Forum Math.}, DOI: 10.1515/forum-2016-0101, to appear.

\bibitem{alesker.survey}
Semyon Alesker. 
\newblock Theory of valuations on manifolds: a survey. 
\newblock \emph{Geom. Funct. Anal.}, 17(4):1321--1341, 2007. 

\bibitem{alesker.faifman}
Semyon Alesker and Dmitry Faifman.
\newblock Convex valuations invariant under the Lorentz group.
\newblock \emph{J. Differential Geom.}, 98(2):183--236, 2014.

\bibitem{artstein.giannopoulos.milman}
Shiri Artstein-Avidan, Apostolos Giannopoulos, and Vitali Milman.
\newblock \emph{Asymptotic geometric analysis. Part I.}
\newblock Mathematical Surveys and Monographs, 202. American Mathematical Society, Providence, RI, 2015.

\bibitem{bernig.survey}
Andreas Bernig. 
\newblock Algebraic integral geometry. 
\newblock In \emph{Global Differential Geometry}, volume 17 of
\emph{Springer Proceedings in Mathematics}, pages 107--145. Springer, Berlin Heidelberg, 2012.

\bibitem{bernig.fu}
Andreas Bernig and Joseph H. G. Fu. 
\newblock Hermitian integral geometry. 
\newblock \emph{Ann. of Math. (2)}, 173(2):907--945, 2011.

\bibitem{bernig.fu.solanes}
Andreas Bernig, Joseph H. G. Fu, and Gil Solanes. 
\newblock Integral geometry of complex space forms. 
\newblock \emph{Geom. Funct. Anal.}, 24(2):403--492, 2014. 

\bibitem{bernig.hug}
Andreas Bernig and Daniel Hug.
\newblock Kinematic formulas for tensor valuations.
\newblock \emph{J. Reine Angew. Math.}, DOI: 10.1515/crelle-2015-0023, to appear.

\bibitem{bianchi.gardner.gronchi}
Gabriele Bianchi, Richard J.~Gardner, and Paolo Gronchi.
\newblock Symmetrization in geometry. 
\newblock \emph{Adv. Math.}, 306:51--88, 2017.

\bibitem{boroczky.ludwig}
Karoly B\"or\"oczky and Monika Ludwig.
\newblock	Minkowski valuations on lattice polytopes.
\newblock \emph{J. Eur. Math. Soc.}, to appear. 

\bibitem{cavallina.colesanti}
Lorenzo Cavallina and Andrea Colesanti.
\newblock Monotone valuations on the space of convex functions.
\newblock \emph{Anal. Geom. Metr. Spaces}, 3:167--211, 2015.

\bibitem{chakerian}
Gulbank Don Chakerian. 
\newblock Inequalities for the difference body of a convex body.
\newblock \emph{Proc. Amer. Math. Soc.}, 18:879--884, 1967.

\bibitem{cianchi_lyz_2009}
Andrea Cianchi, Erwin Lutwak, Deane Yang, and Gaoyong Zhang. 
\newblock Affine Moser-Trudinger and Morrey-Sobolev inequalities.
\newblock \emph{Calc. Var. Partial Differential Equations}, 36(3):419--436, 2009. 

\bibitem{colesanti.hug.saorin}
Andrea Colesanti, Daniel Hug, and Eugenia Saor\'in G\'omez.
\newblock A characterization of some mixed volumes via the Brunn-Minkowski inequality.
\newblock \emph{J. Geom. Anal.}, 24:1064--1091, 2014.

\bibitem{dorrek}
Felix Dorrek.
\newblock Minkowski endomorphisms.
\newblock \emph{Available at} arXiv:1610.08649.

\bibitem{fu.survey}
Joseph H. G. Fu.
\newblock Algebraic integral geometry. 
\newblock In \emph{Integral geometry and valuations}, pages 47--112, Adv. Courses Math. CRM Barcelona, Birkh\"auser/Springer, Basel, 2014. 

\bibitem{gardner.book06}
Richard J. Gardner.
\newblock {\em Geometric tomography}, volume~58 of {\em Encyclopedia of
  Mathematics and its Applications}.
\newblock Cambridge University Press, New York, second edition, 2006.

\bibitem{gardner.hug.weil1}
Richard J. Gardner, Daniel Hug, and Wolfgang Weil.
\newblock Operations between sets in geometry.
\newblock  \emph{J. Eur. Math. Soc.}, 15:2297--2352, 2013.

\bibitem{gardner.hug.weil2}
Richard J. Gardner, Daniel Hug, and Wolfgang Weil.
\newblock The Orlicz-Brunn-Minkowski theory: a general framework, additions, and inequalities.
\newblock  \emph{J. Differential Geom.}, 97:427--476, 2014.

\bibitem{gruber.book}
Peter M. Gruber. 
\newblock \emph{Convex and discrete geometry}, volume 336 of \emph{Grundlehren der Mathematischen Wissenschaften [Fundamental Principles of Mathematical Sciences]}. \newblock Springer, Berlin, 2007.

\bibitem{haberl}
Christoph Haberl. 
\newblock Minkowski valuations intertwining the special linear group. 
\newblock \emph{J. Eur. Math. Soc.}, 14:1565--1597, 2012.

\bibitem{haberl.parapatits}
Christoph Haberl and Lukas Parapatits.
\newblock Valuations and surface area measures. 
\newblock \emph{J. Reine Angew. Math.}, 687:225--245, 2014. 

\bibitem{haberl.schuster1}
Christoph Haberl and Franz E.~Schuster.
\newblock General {$L_p$} affine isoperimetric inequalities.
\newblock \emph{J. Differential Geom.}, 83(1):1--26, 2009.

\bibitem{haberl.schuster2}
Christoph Haberl and Franz E.~Schuster.
\newblock Asymmetric affine {$L_p$} {S}obolev inequalities.
\newblock \emph{J. Funct. Anal.}, 257(2):641--658, 2009.

\bibitem{haberl.schuster.xiao}
Christoph Haberl, Franz E.~Schuster, and Jie Xiao.
\newblock An asymmetric affine P\'olya-Szeg\"o principle. 
\newblock \emph{Math. Ann.}, 352(3):517--542, 2012. 

\bibitem{haddad.jimenez.montenegro}
Juli\'an Haddad, Carlos Hugo Jim{\'e}nez, and Marcos Montenegro.
\newblock Sharp affine Sobolev type inequalities via the $L_p$ Busemann-Petty centroid inequality.
\newblock \emph{J. Funct. Anal.}, 271(2):454--473, 2016. 
   
\bibitem{hadwiger}
Hugo Hadwiger.
\newblock {\em Vorlesungen \"uber {I}nhalt, {O}berfl\"ache und {I}soperimetrie}.
\newblock Springer-Verlag, Berlin, 1957.

\bibitem{kiderlen}
Markus Kiderlen. 
\newblock Blaschke- and Minkowski-endomorphisms of convex bodies.
\newblock {\em Trans. Amer. Math. Soc.}, 358(12):5539--5564, 2006.

\bibitem{klain00}
Daniel A. Klain.
\newblock Even valuations on convex bodies.
\newblock {\em Trans. Amer. Math. Soc.}, 352:71--93, 2000.

\bibitem{klain.rota} 
Daniel A. Klain and Gian-Carlo Rota. 
\newblock \emph{Introduction to geometric probability}.
\newblock Cambridge University Press, Cambridge, 1997.

\bibitem{ludwig02}
Monika Ludwig.
\newblock Projection bodies and valuations.
\newblock \emph{Adv. Math.}, 172:158--168, 2002.

\bibitem{ludwig05}
Monika Ludwig.
\newblock Minkowski valuations. 
\newblock  \emph{Trans. Amer. Math. Soc.}, 357(10):4191--4213, 2005.

\bibitem{ludwig_matrix}
Monika Ludwig.
\newblock Fisher information and matrix-valued valuations. 
\newblock \emph{Adv. Math.}, 226(3):2700--2711, 2011.

\bibitem{ludwig_sob}
Monika Ludwig.
\newblock Valuations on Sobolev spaces. 
\newblock \emph{Amer. J. Math.}, 134(3):827--842, 2012. 

\bibitem{ludwig.reitzner}
Monika Ludwig and Matthias Reitzner.
\newblock A classification of $\SL(n)$ invariant valuations.
\newblock \emph{Ann. of Math. (2)}, 172:1219--1267, 2010.

\bibitem{ludwig.xiao.zhang}
Monika Ludwig, Jie Xiao, and Gaoyong Zhang.
\newblock Sharp convex Lorentz-Sobolev inequalities.  
\newblock \emph{Math. Ann.}, 350(1):169--197, 2011. 

\bibitem{lutwak86}
Erwin Lutwak.
\newblock On some affine isoperimetric inequalities. 
\newblock \emph{J. Differential Geom.}, 23(1):1--13, 1986.  

\bibitem{lutwak.handbook}
Erwin Lutwak.
\newblock Selected affine isoperimetric inequalities.
\newblock \emph{Handbook of convex geometry, Vol.~A}, 151--176. 
\newblock North-Holland, Amsterdam, 1993.

\bibitem{lyz_2000}
Erwin Lutwak, Deane Yang, and Gaoyong Zhang.
\newblock $L_p$ affine isoperimetric inequalities.
\newblock \emph{J. Differential Geom.}, 56(1):111--132, 2000.

\bibitem{lyz_2002}
Erwin Lutwak, Deane Yang, and Gaoyong Zhang. 
\newblock Sharp affine $L_p$ Sobolev inequalities. 
\newblock \emph{J. Differential Geom.}, 62(1):17--38, 2002.

\bibitem{lyz_2010}
Erwin Lutwak, Deane Yang, and Gaoyong Zhang.
\newblock A volume inequality for polar bodies.
\newblock \emph{J. Differential Geom.}, 84(1):163--178, 2010.

\bibitem{mcmullen77}
Peter McMullen.
\newblock Valuations and {E}uler-type relations on certain classes of convex
  polytopes.
\newblock {\em Proc. London Math. Soc. (3)}, 35:113--135, 1977.

\bibitem{mcmullen93}
Peter McMullen.
\newblock Valuations and dissections. 
\newblock \emph{Handbook of convex geometry, Vol.~B}, 933--988. 
North-Holland, Amsterdam, 1993.

\bibitem{mcmullen.schneider}
Peter McMullen and Rolf Schneider.
\newblock Valuations on convex bodies. 
\newblock \emph{Convexity and its applications}, 170--247, Birkh\"auser, Basel, 1983.

\bibitem{milman.rotem}
Vitali Milman and Liran Rotem.
\newblock Characterizing addition of convex sets by polynomiality of volume and by the homothety operation.
\newblock {\em Commun. Contemp. Math.}, 17 (3), 1450022, 22 pp, 2015.

\bibitem{schuster.parapatits}
Lukas Parapatits and Franz E. Schuster.
\newblock The Steiner formula for Minkowski valuations.
\newblock \emph{Adv. Math.}, 230(3):978--994, 2012. 

\bibitem{parapatits.wannerer}
Lukas Parapatits and Thomas Wannerer.
\newblock On the inverse Klain map. 
\newblock \emph{Duke Math. J.}, 162(11):1895--1922, 2013. 

\bibitem{rogers.shephard}
Claude A. Rogers and Geoffrey C. Shephard.
\newblock The difference body of a convex body.
\newblock {\em Arch. Math.}, 8:220--233, 1957.

\bibitem{rogers.shephard58}
Claude A.~Rogers and Geoffrey C.~Shephard.
\newblock Convex bodies associated with a given convex body.
\newblock {\em J. London Math. Soc.}, 33:270--281, 1958.

\bibitem{schneider71}
Rolf Schneider.
\newblock On Steiner points of convex bodies. 
\newblock \emph{Israel J. Math.}, 9:241--249, 1971.

\bibitem{schneider72}
Rolf Schneider.
\newblock Kr\"ummungsschwerpunkte konvexer K\"orper. II. 
\newblock \emph{Abh. Math. Sem. Univ. Hamburg}, 37:204--217, 1972. 

\bibitem{schneider74}
Rolf Schneider.
\newblock Equivariant endomorphisms of the space of convex bodies.
\newblock \emph{Trans. Amer. Math. Soc.}, 194:53--78, 1974.

\bibitem{schneider74.dim2}
Rolf Schneider.
\newblock Bewegungs\"aquivariante, additive und stetige Transformationen
   konvexer Bereiche.
\newblock {\em Arch. Math. (Basel)}, 25:303--312, 1974.

\bibitem{schneider.book14}
Rolf Schneider.
\newblock \emph{Convex bodies: the {B}runn-{M}inkowski theory}, second expanded edition,
\newblock Encyclopedia of Mathematics and its Applications, vol. 151, Cambridge University Press, Cambridge, 2014. 

\bibitem{schneider_schuster06}
Rolf Schneider and Franz E. Schuster.
\newblock Rotation covariant {M}inkowski valuations.
\newblock {\em Int. Math. Res. Not.}, Art. ID 72894, 20 pp., 2006.

\bibitem{schuster.10}
Franz E. Schuster.
\newblock Crofton measures and Minkowski valuations.
\newblock \emph{Duke Math. J.}, 154:1--30, 2010.

\bibitem{schuster.wannerer12}
Franz E. Schuster and Thomas Wannerer. 
\newblock $\GL(n)$ contravariant Minkowski valuations.
\newblock \emph{Trans. Amer. Math. Soc.}, 364:815--826, 2012.
 
\bibitem{schuster.wannerer.smooth}
Franz E.~Schuster and Thomas Wannerer.
\newblock Even Minkowski valuations.
\newblock \emph{Amer. J. Math.},  137(6):1651--1683, 2015.

\bibitem{schuster.wannerer16}
Franz E.~Schuster and Thomas Wannerer.
\newblock Minkowski Valuations and Generalized Valuations.
\newblock \emph{J. Eur. Math. Soc.}, to appear. 

\bibitem{wang}
Tuo Wang. 
\newblock The affine {S}obolev-{Z}hang inequality on {$BV(\mathbb{R}^n)$}.
\newblock \emph{Adv. Math.}, 230(4--6):2457--2473, 2012. 

\bibitem{wannerer.equiv}
Thomas Wannerer.
\newblock $\GL(n)$ covariant Minkowski valuations.
\newblock \emph{Indiana Univ. Math. J.}, 60(5):1655--1672, 2011.

\bibitem{wannerer1}
Thomas Wannerer.
\newblock Integral geometry of unitary area measures.
\newblock \emph{Adv. Math.}, 263:1--44, 2014.

\bibitem{wannerer2}
Thomas Wannerer.
\newblock The module of unitarily invariant area measures.
\newblock \emph{J. Differential Geom.}, 96(1):141--182, 2014.

\bibitem{zhang}
Gaoyong Zhang.
\newblock The affine Sobolev inequality.
\newblock \emph{J. Differential Geom.}, 53(1):183--202, 1999.


\end{thebibliography}
\end{document}